\documentclass[reqno]{amsart} 

\usepackage{graphicx, amsmath, amssymb, amsfonts, amsthm, stmaryrd, amscd,hyperref, bbm}
\usepackage[usenames, dvipsnames]{xcolor}

\usepackage{enumerate}
\usepackage[normalem]{ulem}

\usepackage{ifpdf}

\usepackage{xparse}

\usepackage[all]{xy}
\usepackage{enumerate}

\makeatletter
\newcommand*{\transpose}{%
  {\mathpalette\@transpose{}}%
}
\newcommand*{\@transpose}[2]{%
  \raisebox{\depth}{$\m@th#1\intercal$}%
}
\makeatother

\makeatletter
\newcommand*{\da@rightarrow}{\mathchar"0\hexnumber@\symAMSa 4B }
\newcommand*{\da@leftarrow}{\mathchar"0\hexnumber@\symAMSa 4C }
\newcommand*{\xdashrightarrow}[2][]{%
  \mathrel{%
    \mathpalette{\da@xarrow{#1}{#2}{}\da@rightarrow{\,}{}}{}%
  }%
}
\newcommand{\xdashleftarrow}[2][]{%
  \mathrel{%
    \mathpalette{\da@xarrow{#1}{#2}\da@leftarrow{}{}{\,}}{}%
  }%
}
\newcommand*{\da@xarrow}[7]{%
  \sbox0{$\ifx#7\scriptstyle\scriptscriptstyle\else\scriptstyle\fi#5#1#6\m@th$}%
  \sbox2{$\ifx#7\scriptstyle\scriptscriptstyle\else\scriptstyle\fi#5#2#6\m@th$}%
  \sbox4{$#7\dabar@\m@th$}%
  \dimen@=\wd0 %
  \ifdim\wd2 >\dimen@
    \dimen@=\wd2 %
  \fi
  \count@=2 %
  \def\da@bars{\dabar@\dabar@}%
  \@whiledim\count@\wd4<\dimen@\do{%
    \advance\count@\@ne
    \expandafter\def\expandafter\da@bars\expandafter{%
      \da@bars
      \dabar@ 
    }%
  }%
  \mathrel{#3}%
  \mathrel{%
    \mathop{\da@bars}\limits
    \ifx\\#1\\%
    \else
      _{\copy0}%
    \fi
    \ifx\\#2\\%
    \else
      ^{\copy2}%
    \fi
  }%
  \mathrel{#4}%
}
\makeatother

\usepackage{mathtools}

\theoremstyle{plain} \newtheorem{theorem} {Theorem}  \newtheorem{corollary} [theorem] {Corollary} \newtheorem{proposition} [theorem] {Proposition} 

\theoremstyle{definition} \newtheorem{definition} [theorem] {Definition}

            \newtheorem{remark} [theorem] {Remark}

\newtheoremstyle{itplain} 
{6pt}                    
{5pt\topsep}                    
{\itshape}                   
{}                           
{\itshape}                   
{.}                          
{5pt plus 1pt minus 1pt}                       
{}  

\theoremstyle{itplain} 
\newtheorem{lemma}[theorem]{Lemma}

\newtheorem*{lemma*}{Lemma}
\newtheorem*{proposition*}{Proposition}
\newtheorem*{definition*}{Definition}
\newtheorem*{example*}{Example}

\newtheorem*{results*}{Results}

\usepackage[displaymath,textmath,sections,graphics]{preview}
\PreviewEnvironment{align*}
\PreviewEnvironment{multline*}
\PreviewEnvironment{tabular}
\PreviewEnvironment{verbatim}
\PreviewEnvironment{lstlisting}
\PreviewEnvironment*{frame}
\PreviewEnvironment*{alert}
\PreviewEnvironment*{emph}
\PreviewEnvironment*{textbf}
\PreviewEnvironment*{pn}

\begin{document}


\email{peng-yj21@mails.tsinghua.edu.cn}

\title{The Weyl bound for Rankin-Selberg $L$-functions with Joint Ramification}
\author{Yunjian Peng}

\begin{abstract}
    In this paper, we establish the Weyl bound for the Rankin-Selberg $L$-function in a certain joint ramification setting. To achieve this result, we employ the refined Petersson trace formula and develop a special Voronoï summation formula. Additionally, we obtain the sharp bound for the integral of products of Whittaker functions via the $p$-adic stationary phase method.
\end{abstract}

\maketitle


\tableofcontents
\section{Introduction} 

Let $\Pi$ be an automorphic representation of a reductive group $G$ and $L(\Pi,s)$ be the $L$-function associated to $\Pi$. The generalized Lindel\"{o}f conjecture, which is a corollary of the Riemann Hypothesis for $L(\Pi,s)$,  says that for any positive $\epsilon$ we have 
\[
    L(\Pi,1/2) \ll C(\Pi)^\epsilon
\]
where $C(\Pi)$  is the analytic conductor of $\Pi$ which is a product over all places of the local analytic
conductors $C_\nu (\Pi)$. The Phragm\'en--Lindel\"of principle implies the bound 
\begin{equation}\label{convexity}
    L(\Pi,1/2) \ll C(\Pi)^{1/4+\epsilon}
\end{equation}
which is commonly referred to as the convexity bound.
The subconvexity problem, in its most general form, is the problem of improving upon (\ref{convexity}) in the exponent.

While there are very few results on the subconvexity problem for general $L$-functions, there is a huge amount of literature dealing with special cases and special families. For example, the subconvexity problem for automorphic $L$-functions of $\mathrm{GL}_2$ over number fields has been solved completely, with non-specific exponent, in the groundbreaking work \cite{MichelVenkatesh2010Subconvexity}. They also study $\mathrm {GL}_2\times \mathrm {GL}_2$ case and obtain the subconvex bound when one representation is fixed
\[
    L(1/2,\pi_1\times \pi_2) \ll_{\pi_2} C(\pi_1\times \pi_2)^{1/4-\delta}.
\]

In the level aspect, let $\pi_1$ and $\pi_2$ be cuspidal automorphic representations
 both with trivial central character whose archimedean components are bounded. Hu, Michel and Nelson prove in \cite{HMN23} that
\[
    L(1/2, \pi_1\times \pi_2)\ll C(\pi_1\times \pi_2)^{1/4-\delta }
\]
 when $\min_{i=1,2} \{ C(\pi_i\times \tilde \pi_i) \}\le C(\pi_1\times \pi_2 )^{1-\gamma }$ for some fixed $\gamma >0$. 

On the other hand,  there are some works  that establish strong subconvex bounds for special automorphic $L$-functions.For instance, let $f$, $g$ be holomorphic cusp newforms of level $M$, $N$ respectively. Holowinsky and Templier prove in \cite{Holowinsky2014} that 
\[
L(1/2,f\times g) \ll C(f\times g)^{1/6+\epsilon}
\]
for $M$, $N$ coprime, square free and $M = N^{1/2+o(1)}$ with assumption that $L(1/2, f\times g)\ge 0$. The exponent $1/6$ is referred to as the Weyl bound due to Weyl who first obtained this bound for $\zeta$ function in \cite{Wey21}. Recently Yueke Hu \cite {Hu23} have obtained the Weyl bound using a refined Petersson trace formula. More precisely, fix a prime number $p\ne 2$. Let $M=p^\mathfrak c$ with $\mathfrak c\ge 3$, $N$ square-free which is coprime to $M$. Assume that $L(f\times g,1/2)\ge 0$ for all $f\in \mathcal F_\theta[l]$ (Definition \ref{small}) and all $l$. Then if $\sqrt N\le M \le N^2$, it has been shown that
\[
   L(1/2,f\times g) \ll_{\kappa,\iota, \epsilon,p} (MN)^{1/3+\epsilon}.
\]

In this paper, we will prove the following main theorem. 

\begin{theorem}\label{Thm1}
Let $M=p^{\mathfrak c_1}$ with $\mathfrak c_1\ge 3$ and  Let $g$ be a holomorphic cuspidal newform with level $N= p^{\mathfrak c_2}$, fixed weight $\iota \ge 4$. Then 
\begin{equation}\label{firstmoments}    
\sum_{f\in \mathcal F_\theta [l]} \frac  {L(1/2, f\times g)} {\Vert f \Vert ^2} \ll_{p,\varepsilon} N^{\varepsilon} \left (p^{\mathfrak c_1/2+l} + \frac {p ^{\mathfrak c_2/2 }} {p^{\mathfrak c_1/4+l/2}}\right ).
\end{equation}
Furthermore  suppose $L(f\times g,1/2) \ge 0$ for all $f\in \mathcal F_\theta [l]$. Suppose $\mathfrak c_1 = \delta \mathfrak c_2$ for $0<\delta<1$. By picking $l= l_0 \in \{ 0, 1\}$, we get the subconvex bound 
\begin{equation*}
L(1/2,f\times g) \ll N^{\max \{ \frac \delta  2, \frac {1- \delta }  2  \}+\epsilon}.
\end{equation*}
By picking  $l$ to be the closest integer to $3\mathfrak c_2/2 -\mathfrak c_1/2$
  while $1\le l <i_0$, we get that
\begin{equation*}
 L(1/2, f\times g) \ll_{p,\varepsilon} N^{\max \left \{\frac {1-\delta} 2, \frac 1 3, \frac \delta 2 \right \}+ \varepsilon}.
\end{equation*}
  In particular we obtain a hybrid subconvexity bound for $\delta$  in any compact subset of $(0,1)$, which is further more a Weyl bound in the range $1/3 \le  \delta  \le  2/3$.
\end{theorem}

\begin{remark}
The condition that $L(1/2,f\times g) \ge 0$ for all $f$ can be guaranteed when, for example, $g$ is dihedral. See the discussion in \cite[Sect. 1.1]{Holowinsky2014}.
\end{remark}
We will prove this in Section \ref{section5}. Here we briefly outline the proof strategy. We adopt the notation of Theorem \ref{Thm1}. Assume that $f$ and $g$ correspond to automorphic forms $\pi_1 = \otimes_\nu  \pi _{1,\nu }$, $\pi_2 = \otimes_\nu \pi _{2,\nu}$. First, by applying the approximate functional equation (\ref{Approxomate}) to the left side of (\ref{firstmoments}), we expand it into sums involving the Fourier coefficients of $f$ and $g$. This results in two types of sums: one without the $\epsilon$-factor and one with the $\epsilon$-factor. In Section \ref{epsilonfactors}, we prove that when $f \in  \mathcal F_\theta[l]$, the $\epsilon$-factor remains unchanged, allowing us to handle both sums simultaneously. Then, we interchange the order of summation and use the refined trace formula (Theorem \ref{trace formula}) to transform $\sum_{f \in  \mathcal F_\theta [l]} \lambda _f(n)$ into a sum of Generalized Kloosterman sums $G(n,1, \theta ,1/c^2)$, which are associated with the local Whittaker function of $\pi _{1,p}$. Next, for the summation over $n$, we use a special Voronoi summation formula (Theorem \ref{Voronoi}), which is associated with the local Whittaker function of \(\pi_{2,p}\). Eventually, the problem reduces to the estimation of the integral below:
\[
\mathcal P =\int _{\mathfrak o^\times } W^{(i)} _{\pi_1,p} (x) W^{(j)} _{\pi_2,p} (x) \psi \left (\frac {x} {\varpi ^l}\right )dx,
\]
where $\psi $ is  an additive character defined in Section \ref{preliminaries}, and  $W_{\pi_1, p} ^{(j)}$ is  a value  of Whittaker function for newform along certain double coset, which is defined in Definition \ref{definitionW^k}. In Section \ref{section3}, we study this integral carefully and obtain the square root cancellation using $p$-adic stationary method repeatedly.

There are two novelties in our article. Firstly, to study the first moments of $L$-functions via approximate functional equation, it is necessary to make sure that $\epsilon (f\times g)$ are same while $f$ belongs to some family, which was trivial in the past. In our case, we need to calculate $\epsilon (\pi_1 \times \pi_2)$ by the local Langlands correspondence and we find that even when the level of $f$ is fixed, the value of the $\epsilon$-factor changes as $f$ varies within the cusp newforms (see the proof of Proposition \ref{epsilon_factor_independence}). Thus, the classical Petersson trace formula no longer applies in this case and it is essential to introduce a small family of newforms and the refined Petersson trace formula.

Secondly, the square root cancellation of the local integral $\mathcal P$ is a nontrivial result which plays the similar role as \cite[Proposition 5.4]{HMN23}. This  integral can be viewed as a generalization of matrix coefficients as in \cite[(34)]{HS20} or $\widetilde G_p$ studied in \cite[Sect. 5.3]{Hu23}. We believe that the method we used can be applied to other local integrals about Whittaker function and other joint ramification cases in the future.

\section{Preliminaries}\label{preliminaries}
\subsection{Notation and prerequisites}
For any $x\in \mathbb R$, let $\lceil x \rceil$  be the least integer greater than or equal to $x$, and $\lfloor x \rfloor$ be the greatest integer less than or equal to $x$.

Globally we work with the rational field $\mathbb Q$. Let $\mathbb A$ be the ring of adeles over $\mathbb Q$. We fix an additive character $\psi$  on $\mathbb A / \mathbb Q$, which is a product of local additive characters $\psi _\nu$, where $\psi _\infty (x) = e^{-2\pi ix}$, and $\psi _p(x) = e^{2\pi ix'}$ where $x' \in \mathbb Q$ and $x' \equiv  x \mod {\mathfrak o_p}$.

Let $F$ denote a $p$-adic local field, $\mathfrak o=\mathfrak o_F$ be its ring of integers and $\varpi = \varpi _F$ be a uniformizer with order of residue field $p\ne 2$.  In general the level of an additive character $\psi _\nu$  is defined to be the smallest integer $\mathfrak c(\psi _\nu )$ such that $\psi _\nu$  is trivial on $\mathfrak p ^{\mathfrak c(\psi _\nu)} $. Let $U(n)= U_F(n) = 1+\mathfrak p^n$ when $n\ge 1$, and $U(0)= U_F(0) = \mathfrak o^\times$. The level of a multiplicative character $\chi$ is defined to be the smallest non-negative integer $\mathfrak c(\chi )$ such that $\chi$ is trivial on $U(\mathfrak c(\chi))$. 

Further, we equip the local field $F$ with two measures. First of all, we consider the Haar measure $\mu = \mu _F$  on $(F ,+)$ which normalized such that $\mu (\mathfrak o_p )=1$. The second measure is the Haar measure $\mu  ^\times$ on
 $(F ^\times ,\times )$ which is explicitly given by $\mu  ^\times = \frac {\zeta _F (1)} {|\cdot |_F } \mu $, where $\zeta _F (s)=(1-|\mathfrak o/\mathfrak p| ^{-s})^{-1}$ denotes the local Euler factor of the Riemann zeta function. In particular, one has $\mu  ^\times (\mathfrak o^\times )= 1$. At the archimedean place $\infty$ we take $\mu _\infty$ to be the standard Lebesgue measure and  simply choose $\mu_\infty ^\times = \mu _\infty |\cdot |^\infty$.

Let $E$ be a quadratic \'{e}tale algebra over $F$. When $E$ is a field, let $e_E$ be the ramification index of $E$. Let $\mathfrak o_E$, $\varpi _E$, $U_E(n)$, $\mu _E$ and $\mu_E^\times$ be defined similarly as for $F$. 

If $E = F\times F$ splits, let $e_E = 1$. Let $U_E(0) = \mathfrak o_E^\times  = \mathfrak o^\times \times \mathfrak o^\times$, and $U_E(n) = 1+\varpi ^n(\mathfrak o \times \mathfrak o$). Let $\mu _E$ and $\mu_E^\times$ be the product measure.

 Let $\theta$ be a character of $E^\times$. Let $\mathfrak c(\theta )$ be the level of $\theta $. If $E$ splits, we assume that $\theta |_{F^\times } =1$. Then we can write
 $\theta  = (\chi ,\chi^{-1}) $, and define $\mathfrak c(\theta ) = \mathfrak c(\chi)$.

Let $R$ be a commutative ring with $1$. In our case,$R$ will be either $\mathbb Q$, $\mathbb Q_p$, $\mathbb R$, or $\mathbb A$. We set $G(R)=\mathrm {GL}_2(R)$ and define the subgroups
\begin{equation*}
\begin{split}
&A(R) = \left \{ a(r)= \left (
\begin{array}{cc}
    r & 0 \\
    0 & 1
\end{array}
\right) : r\in R^\times  \right \}, \ 
N(R) = \left \{ n(r)=\left (
\begin{array}{cc}
    1 & r \\
    0 & 1
\end{array}
\right): r\in R \right \}, \\ 
&Z(R) = \left \{ z(r)=\left (
\begin{array}{cc}
    r & 0 \\
    0 & r
\end{array}
\right): r\in R^\times  \right \}.
\end{split}
\end{equation*}
Over $F$, let $K$ be the standard maximal compact
 open subgroup $GL_2(\mathfrak o)$ and 
 \begin{equation*}
 K_0(\mathfrak c) = \left \{ g=\left (
\begin{array}{cc}
    a & b \\
    c & d
\end{array}
\right)\in K: c\in \mathfrak p^{\mathfrak c}  \right \}.
\end{equation*}

Let $\pi$ be an irreducible cuspidal automorphic representation of $\mathrm {GL}_2$ over $\mathbb Q$ with trivial central character. Let $\pi_p$  denote its local component at finite place $p$. Let $\mathfrak c(\pi _p)$ be the level of $\pi _p$, which is the smallest integer such that $\pi _p$ contains an element invariant by $K_0(\mathfrak c(\pi _p))$.

We denote the Weyl element by 
\[
\omega = \left (
\begin{array}{cc}
    0 & 1 \\
    -1 & 0
\end{array}
\right).
\]
For an automorphic cuspidal form  $\phi$, define 
\[
    \| \phi \| = \langle \phi, \phi \rangle = \int _{Z(\mathbb A) \mathrm {GL}_2(\mathbb Q) \backslash \mathrm {GL}_2(\mathbb A)} | \phi (g) |^2 \mathrm d g.
\]
\subsection{Rankin-Selberg \texorpdfstring{$L$}{L}-function}\label{subsetion}
Here, we recall some fundamental facts concerning Rankin-Selberg $L$-functions, with our primary reference being \cite{Mic04}.  Given $f$ and $g$ two holomorphic cusp newforms with  nebentypus $\chi_f$, $\chi_g$  respectively, the Rankin-Selberg $L$-function is a degree four Euler product
\begin{equation*}
L(f\times g, s) = \sum_{n\ge 1} \frac {\lambda _{f \times g} (n)  } {n^s} .
\end{equation*}
Its standard analytic properties (analytic continuation, functional equation) have been known for a while (from the work of Rankin, Selberg and others, see \cite{Jac72},  \cite{JPSS81}).

We set $N = \sqrt {C(f\times g)}$ and 
\begin{equation*}
G(u) = \left (\cos{ \frac {\pi u} {4A_0}}\right )^{-5A_0},
\end{equation*}
for $A_0 \ge 1$. By a contour shift we infer from functional equation that for Re $s = 1/2$,
\begin{equation}\label{Approxomate}
L(s,f\times g) = \sum _{n\ge 1} \frac {\lambda _f(n) \lambda _g(n)} {n^s} W_s\left (\frac {n} {N}\right ) + \omega _{f\times g} (s) \sum _{n\ge 1} \frac {\overline { \lambda _f(n)  \lambda _g(n)}} {n^{1-s}} \widetilde {W}_{1-s}\left (\frac {n} {N}\right ),
\end{equation}
where 
\begin{equation*}
\begin{split}
& \omega_{f\times g} (s)= \epsilon (f\times g)N^{1 - 2s} \frac {L_\infty (1-s, f\times g )}{L_\infty (s,f\times g )},\\
& W_s(y) = \sum _{d\mid D^\infty} \frac {\gamma_{f\times g}(d)} {d^s} V_s (dy),\\
& V_s(y) = \frac 1 {2\pi i} \int _{(1)} \frac {L_\infty (s+u, f\times g )}{L_\infty (s,f\times g )}L(\chi_f \chi_g, 2s+2u)\frac {G(u)} u y^{-u} du ,
\end{split}
\end{equation*}
and  $\widetilde W_s$ is defined like $W_s$ except that $\gamma _{f\times g} (d)$ and $\chi_f\chi_g$ are replaced by $\overline{\gamma_{f\times g} (d)}$ and $\overline{\chi_f\chi_g}$.
\begin{lemma}
When $Re \ s = 1/2$ and $\chi_f \chi_g$ is trivial we have 
\begin{equation*}
W_s(y)\ll _{s,A_0} (1+y/P)^{-A_0} \log (1+y^{-1}).
\end{equation*}
\end{lemma}

\begin{proof} 
See \cite[Lemma 3.1, Remark 3.3]{Mic04}.
\end{proof}
\subsection{Compact induction and local Langlands correspondence} 
Now, we will briefly explain that the representations of \(\text{GL}_2(F)\) can be parameterized in two ways: one via induction from the small group, and the other via the Langlands correspondence. We  will also briefly introduce their relationship. Let $F = \mathbb Q_p$ with $p \ne 2$. Let $E$ be a quadratic extension of $F$, and $\theta$ be a character of $E$ ,where $e_E$ denotes the ramification index of $E$ over $F$, and $c(\theta )$ is the level of $\theta$. Then there exists a supercuspidal representation $\pi _\theta$ of $\mathrm {GL}_2(F)$ constructed via compact induction. Moreover, all supercuspidal representations arise in this manner (for an explicit construction,  refer to \cite{BushnellHenniart2006}. When $\pi_\theta$ has a trivial central character,  it follows that  $\theta |_{F^\times } =1$. In this paper, we utillize compact induction to characterize  small families and Whittaker models.

On the other hand, let $W_F$ be the absolute Weil group of the local field $K$ (see \cite{Tat79} for properties and further references). The modified Weil-Deligne group is the product group $DW_F =\mathrm {SL}_2 (\mathbb C) \times W_F$. Consider continuous finite dimensional complex semi-simple representations $\sigma$ of this group which are analytic when restricted to $\mathrm {SL}_2 (\mathbb C )$. For the purposes of this paper it will only be necessary to consider representations trivial on $\mathrm {SL}_2(\mathbb C)$.

The class field theory isomorphism of $W_F^{ab}$ with $F^\times$  will be normalized so that the geometric Frobenius element corresponds to a generator of $\mathfrak p$. With this chosen isomorphism the characters of $F^\times$  will be identified with continuous one-dimensional complex representations of $W_F$. In particular, nontrivial quadratic characters of $F^\times$  correspond to quadratic separable extensions $E$ of $F$. 
\begin{definition}
The symbol $\omega _{E/F}$ will denote the nontrivial character of $F^\times$  with kernel equal to the group of norms from $E^\times$  to $F^\times $.
\end{definition}
 
\begin{definition} 
 When $E$ is an extension of $F$, the Weil group $W_E$ is a finite index subgroup of $W_F$, and similarly for the modified Weil-Deligne groups. The restriction of $\sigma$ to $DW_E$  will be denoted by $\sigma _E$. 
\end{definition} 

\begin{lemma}
For one-dimensional representations $\theta$  of $W_F$, $\theta _E$ corresponds to $\theta \circ  N_{E/F}$, where $N$ is the norm map. 
\end{lemma}

\begin{definition}
For a representation \(\eta\) of \(W_E\), we denote the induced representation of \(W_F\) from \(\eta\), i.e., \(\mathrm{Ind}_{W_E}^{W_F} \eta\), by \(\mathrm{Ind}_E^F \eta\). For an irreducible, admissible representation \(\pi\) of \(\mathrm{GL}_2(F)\), let \(\sigma(\pi)\) stand for the two-dimensional representation of \(DW_F\) associated to \(\pi\) via the local Langlands correspondence (see \cite{Kut80}).
\end{definition}
Next, we will present the specific content of the $p$-adic local Langlands correspondence for $\mathrm {GL}_2$.

\begin{proposition}
For the irreducible principal series $\pi =\mathrm {Ind} (\chi_1\otimes \chi_2)$, $\sigma (\pi)$ is a sum of two characters $\chi_1$, $\chi_2$. For the supercuspidal represetation $\pi$  of $\mathrm {GL}_2(F)$, $\sigma (\pi)$ is obtained by inducing a character $\theta$ of a separable quadratic extension $E$. Then a representation $\pi$ of $\mathrm{GL}_2(F)$ can be associated with a character $\theta$ over a quadratic étale algebra $E$ over $F$ by local Langlands correspondence.
\end{proposition}

For a quadratic field extension of $E/F$, let $\lambda (\psi )$ be the Langlands $\lambda$-function 
\begin{definition}(1) If $E$ is inert, define $\Delta_\theta$ to be  the unique unramified character of $E^\times$  of order $2$.

(2) If $E$ is ramified and $c(\theta)$ is even associate $\alpha_\theta$ to $\theta$ as in Lemma \ref{character}. Then define $\Delta_\theta$ to be the unique level $1$ character of $E^\times$ such that
\[
 \Delta _\theta |_{F^\times } = \omega _{E/F}, \ \Delta _\theta (\varpi _E) = \omega _{E/F} (\varpi _E^{c(\theta )-1} \alpha _\theta ) \lambda _{E/F} ^{c(\theta)-1}(\psi ).
\]
\end{definition}

\begin{theorem}[\cite{BushnellHenniart2006}] For a representation \(\pi\) that is obtained by compact induction from a character \(\theta\) over the quadratic extension E, its corresponding Deligne-Weil representation under the local Langlands correspondence satisfies \(\sigma(\pi) = \mathrm{Ind}_E^F (\theta \Delta_\theta^{-1})\).
\end{theorem}

\begin{proposition}\label{antiinvariant}
Let $E/F$ be a quadratic field extension and $\tau$ be a nontrivial element of $\mathrm {Gal} (E/F)$. Let $\theta$ be a character of $E^\times$ such that $\theta |_{F^\times } =1$. Set $\Theta = \theta \Delta _\theta ^{-1}$. Then we have $\Theta ^\tau = \Theta ^{-1}$.
\end{proposition}

\begin{proof}
Let $x\in E^\times$, then by the definition $\theta (x) \theta ^\tau (x) = \theta (x\tau (x)) =1$ . Thus $\theta ^\tau = \theta ^{-1}$ and $\Delta _\theta ^\tau = \Delta ^{-1}$.  If $E$ is inert, by the uniqueness of $\Delta _\theta$ we have $\Delta _\theta = \Delta _\theta ^\tau = \Delta _\theta ^{-1}$. If $E$ is ramified, then $\Delta _\theta (x) \Delta _\theta (\tau (x)) = \Delta _\theta (x\tau (x))=\omega _{E/F}(N_{E/F}(x))=1 $. Thus the statement holds.
\end{proof}

\subsection{Small families and the refined trace formula}
In this subsection, we state the refined trace formula as developed in \cite{Hu23}.
\begin{definition}\label{small}Let $E$ be an \'{e}tale quadratic algebra over $F$. Let \(\theta_i\) (\(i = 1, 2\)) be characters of \(E^\times\) satisfying \(\theta_i|_{F^\times} = 1\) and \(\mathfrak{c}(\theta_1) = \mathfrak{c}(\theta_2) = \mathfrak{c}\). Let \(i_0 = \mathfrak{c}(\theta)/e_E\). For \(0 \leq l < i_0\), we write \(\theta_1 \sim_l \theta_2\) if \(\mathfrak{c}(\theta_1^{-1}\theta_2) \leq e_E n\). For a fixed character \(\theta\) with \(\theta|_{F^\times} = 1\), let \(\theta[l] = \{ \theta' \mid \theta' \sim_l \theta,\  \theta|_{F^\times} = 1 \}\). We define  
\[
\pi_\theta[n] = \{ \pi_{\theta'} \mid \theta' \in \theta[n] \}.
\]
For a fixed even weight \(\kappa \ge 4\), we define \(\mathcal{F}_\theta[l]\) as the subset of holomorphic newforms with weight \(\kappa\), level \(p^{\mathfrak{c}}\) (where \(\mathfrak{c} \ge 3\)), and trivial nebentypus, whose associated local representation \(\pi_p\) lies in the family \(\pi_\theta[l]\).
\end{definition}

\begin{theorem} [\cite{Hu23} Theorem 1.7] \label{trace formula}
For \(1 \le l < i_0\), we have
\begin{equation*}
\begin{split}
&\sum_{f \in \mathcal{F}_\theta[l]} \frac{1}{\|f\|^2} \lambda_f(m_1) \overline{\lambda_f(m_2)}\\
&= 
C_{\mathcal{F}}[l] \frac{(4\pi)^{\kappa-1}}{(\kappa - 2)!} \bigg( \delta_{m_1 = m_2} + 2\pi i^\kappa \sum_{\substack{c_l \mid c \\ c > 0}} \frac{G(m_1, m_2, \theta, c^{-2})}{c} J_{\kappa - 1} \bigg( \frac{4\pi \sqrt{m_1 m_2}}{c} \bigg) \bigg),
\end{split}
\end{equation*}
where
\[
\begin{split}
c_l &\asymp  p^{\mathfrak{c}/2+l}, \\
C_{\mathcal{F}}[l] &= p^\mathfrak{c} \, [\theta[l] : \theta[0]] \, D_{\mathcal{F}} \asymp_p p^{\mathfrak{c}/2 + l}.
\end{split}
\]
\(G(m_1, m_2, \theta, c^{-2})\) denotes the generalized Kloosterman sum defined in Lemma \ref{generalkloosterman}, and \(J_{\kappa - 1}\) is the Bessel function of the first kind.
\end{theorem}
Let \( W_{\pi,p} \) be the Whittaker function of the new vector, normalized by \( W_{\pi,p}(1) = 1 \). By \cite[page 39]{Hu23}, we have the following result:

\begin{lemma}\label{generalkloosterman}
Let \( c = p^k d \) with \( (d, p) = 1 \). Then
\[
\begin{aligned}
&G\left( n, 1, \theta, \frac{1}{c^2} \right) = \frac{a_{\pi} p^k}{C_{\mathcal{F}}[l_0] (1 - p^{-1})} \\
& \times \sum_{y \in (\mathbb{Z}/d\mathbb{Z})^\times} e\left( \frac{\bar{p}^{2k} y + n \bar{y}}{d} \right) 
\int_{\mathfrak{o}_p^\times} W_{\pi,p} \left( a\left( \frac{1}{c^2} \right) \omega n\left( \frac{u}{p^k} \right) \right) e_p\left( -\frac{n u}{p^k} \right) du,
\end{aligned}
\]
where the constant satisfies \( a_{\pi} p^k C_{\mathcal{F}}[l_0]^{-1} (1 - p^{-1})^{-1} \asymp p^k \).
\end{lemma}

\subsection{Local \texorpdfstring{$\epsilon$-factors}{epsilon-factors}}\label{epsilonfactors}
In this subsection, we aim to prove that $\epsilon (f\times g)$ appears in (\ref{Approxomate}) is invariant when $f \in \mathcal F_\theta [l]$. Since the 
 $\epsilon$-factor is determined by local factors, and only the local component of $f$ at $p$ varies, it is sufficient to focus on $\epsilon (\pi _{1,p}\times \pi _{2,p})$. Since all subsequent discussions are local, we assume $p\ne 2$ and ignore the subscript $p$.
 
 For each irreducible representation \(\pi\) of \(\mathrm{GL}_2(F)\), Godement and Jacquet introduced the \(L\)-function \(L(\pi)\), \(\gamma\)-factor \(\gamma(\pi)\), and \(\epsilon\)-factor \(\epsilon(\pi)\) in their foundational work on automorphic forms and \(L\)-functions \cite{Godement1972Zeta}. For two representations \(\pi_1\) and \(\pi_2\) of \(\mathrm{GL}_n(F)\), the Rankin-Selberg method provides a way to define the corresponding constants \(L(\pi_1 \times \pi_2)\), \(\gamma(\pi_1 \times \pi_2)\), and \(\epsilon(\pi_1 \times \pi_2)\) via integral representations \cite{JPSS81}. However, the explicit computation of \(\epsilon\)-factors via this approach remains notoriously challenging due to the complexity of evaluating these integrals in practice \cite{Bump1998Automorphic}.

The Local Langlands Correspondence (LLC) offers a more tractable approach to computing \(\epsilon\)-factors. For a non-trivial additive character \(\psi\) of \(F\) and a virtual representation \(\sigma\) of the Weil-Deligne group \(DW_F\), one associates a local \(\epsilon\)-factor \(\epsilon(\sigma, \psi) \in \mathbb{C}^\times\). Specifically:

- For an unramified character \(\chi\) of \(F^\times\), the \(\epsilon\)-factor is given by \(\epsilon(\chi, \psi) = \chi(\varpi)^{c(\psi)}\), where \(c(\psi)\) denotes the conductor of \(\psi\) and \(\varpi\) is a uniformizer of \(F\).

- For a ramified character \(\chi\) of \(F^\times\), the \(\epsilon\)-factor is defined as:
  \[
  \epsilon(\chi, \psi) = \frac{\int_{F^\times} \chi(x^{-1}) \psi(x) \, d^\times x}{\left| \int_{F^\times} \chi(x^{-1}) \psi(x) \, d^\times x \right|},
  \]
  where \(d^\times x\) denotes the multiplicative Haar measure on \(F^\times\).

\(\epsilon\)-factors satisfy the following fundamental properties, see for example, \cite{Tun83}:

\begin{description}
\item [Additivity under direct sums] \(\epsilon(\sigma_1 \oplus \sigma_2, \psi) = \epsilon(\sigma_1, \psi) \epsilon(\sigma_2, \psi)\).

\item [Functoriality for induction] \(\epsilon(\mathrm{Ind}_E^F \sigma, \psi) = \epsilon(\sigma, \psi_E)\) where \(\sigma\) is a degree-zero representation of \(DW_E\) and \(\psi_E(x) = \psi(T_{E/F}(x))\) with \(T_{E/F}\) denoting the trace map from \(E\) to \(F\).

\item [Duality relation] \(\epsilon(\sigma, \psi) \cdot \epsilon(\sigma^*, \psi) = \det \sigma(-1)\) where \(\sigma^*\) is the contragredient representation of \(\sigma\).

\item [Compatibility with tensor products over quadratic extensions] For a quadratic extension \(E/F\), let \(\sigma\) be an even-dimensional representation of \(W_E\) and \(\theta\) a character of \(E^\times\). Then 
   \[
   \epsilon\left(\sigma \otimes \mathrm{Ind}_E^F \theta, \psi\right) = \epsilon\left(\sigma|_E \otimes \theta, \psi_E\right) \cdot \omega_{E/F}^{\dim \sigma/2}(-1),
   \]
   where \(\omega_{E/F}\) denotes the quadratic character associated to \(E/F\) via local class field theory.

\item [Galois invariance:] For a Galois extension \(E/F\) and \(\tau \in \mathrm{Gal}(E/F)\), let \(\sigma^\tau\) denote the representation of \(W_E\) which is composition of  \(\sigma\) and \(\tau\). Then \(\epsilon(\sigma, \psi_E) = \epsilon(\sigma^\tau, \psi_E)\).
\end{description}

The LLC relates the \(\epsilon\)-factors of group representations to those of Galois representations via:
\[
\begin{aligned}
\epsilon(\pi, \psi) &= \epsilon(\sigma(\pi), \psi), \\
\epsilon(\pi_1 \times \pi_2, \psi) &= \epsilon(\sigma(\pi_1) \otimes \sigma(\pi_2), \psi),
\end{aligned}
\]
where \(\sigma(\pi)\) denotes representation of \(DW_F\) associated to \(\pi\) via the LLC.

\begin{theorem}[\cite{Tun83} Theorem 1.4]\label{epsilon}
 Let $E$ be a quadratic separable
 extension of $K$, and let $\theta$ be a character of $E^\times$  which is trivial on $F^\times$. Let $\Delta$ be an element of $E^\times$ with trace to $F$ equal to zero. Then $\epsilon(\theta, \psi_E) = c\theta (\Delta)$, where $c$ is a constant independent of $\theta$.
\end{theorem}

\begin{proposition}[\cite{Tun83}]
Let \( E \) and \( L \) be two distinct quadratic extensions of \( F \), and let \( \theta \) be a character of \( E \), viewed as a character of the Weil group \( W_E \). Then we have
\(
\left( \mathrm{Ind}_E^F \theta \right) \big|_L = \mathrm{Ind}_{EL}^L \left( \theta \circ N_{EL/L} \right).
\)

For a quadratic extension \( E/F \), let \( \eta \) (resp. \( \theta \)) be a character of \( F^\times \) (resp. \( E^\times \)), viewed as a character of \( W_F \) (resp. \( W_E \)). Then we have
\(
\eta \otimes \mathrm{Ind}_E^F \theta = \mathrm{Ind}_E^F \left( \theta \otimes \left( \eta \circ N_{E/F} \right) \right).
\)
\end{proposition}

\begin{proposition}[\(\epsilon\)-factor independence]\label{epsilon_factor_independence}
Suppose $p\ne 2$. Let \(\pi_1\) and \(\pi_2\) be two irreducible representations of \(\mathrm{GL}_2(F)\) with trivial central characters, whose conductors satisfy \(\mathfrak{c}(\pi_1), \mathfrak{c}(\pi_2) \ge 2\). Suppose these representations correspond to the pairs \((E, \eta)\) and \((L, \theta)\) via compact induction or parabolic induction. Then the \(\epsilon\)-factor \(\epsilon(\pi_1 \times \pi_2, \psi)\) depends only on the extensions \(E, L\) and the additive character \(\psi\), and is independent of the choices of \(\eta\) and \(\theta\).
\end{proposition}
\begin{proof}
For simplicity, we omit the standard additive character $\psi$. Write \(\sigma_1 = \sigma(\pi_1)\) and \(\sigma_2 = \sigma(\pi_2)\). We prove the proposition case by case.

\textbf{Case 1} Suppose \(\pi_1 = \mathrm{Ind}\chi_1^{-1} \otimes \chi_1\) and \(\pi_2 = \mathrm{Ind}\chi_2^{-1} \otimes \chi_2\) are irreducible principal series representations. Then we have
\begin{equation*}
\epsilon(\pi_1 \times \pi_2) = \epsilon\left( (\chi_1^{-1} \oplus \chi_1) \otimes (\chi_2^{-1} \oplus \chi_2) \right) = \epsilon(\chi_1^{-1}\chi_2^{-1}) \cdot \epsilon(\chi_1^{-1}\chi_2) \cdot \epsilon(\chi_1\chi_2^{-1}) \cdot \epsilon(\chi_1\chi_2) = 1.
\end{equation*}
The last equality holds because \(\epsilon(\chi_1^{-1}\chi_2^{-1}) \cdot \epsilon(\chi_1\chi_2) = \chi_1(-1)\chi_2(-1)\) and \(\epsilon(\chi_1^{-1}\chi_2) \cdot \epsilon(\chi_1\chi_2^{-1}) = \chi_1^{-1}(-1)\chi_2(-1)\).

\textbf{Case 2} Assume only one of \(\pi_1\), \(\pi_2\) is a principal series representation. Let \(\sigma_1 = \chi^{-1} \oplus \chi\) and \(\sigma_2 = \mathrm{Ind}_L^F(\Theta)\) where \(\Theta = \theta \Delta_\theta^{-1}\). Then
\begin{equation*}
\epsilon(\pi_1 \times \pi_2) = \epsilon\left( (\chi^{-1} \oplus \chi) \otimes \mathrm{Ind}_L^F(\Theta) \right) = \epsilon\left( \chi^{-1}|_L \otimes \Theta, \psi_L \right) \cdot \epsilon\left( \chi|_L \otimes \Theta, \psi_L \right)
\end{equation*}
Let \(\tau\) be the nontrivial element of \(\mathrm{Gal}(L/F)\). Note that \(\chi|_L = \chi \circ N_{L/F}\) is \(\tau\)-invariant, and by Proposition \ref{antiinvariant}, \(\Theta^\tau = \Theta^{-1}\). Thus we have
\begin{equation*}
\epsilon\left( \chi|_L \otimes \Theta, \psi_L \right) = \epsilon\left( \chi|_L \otimes \Theta^{-1}, \psi_L \right).
\end{equation*}
It follows that \(\epsilon\left( \chi^{-1}|_L \otimes \Theta, \psi_L \right) \cdot \epsilon\left( \chi|_L \otimes \Theta^{-1}, \psi_L \right) = \chi(N_{L/F}(-1)) \Theta^{-1}(-1) = \Delta_\theta(-1)\), which is independent of \(\theta\) and \(\eta\).

\textbf{Case 3} Now consider the case where both \(\pi_1\) and \(\pi_2\) are supercuspidal. Let \(\sigma_1 = \mathrm{Ind}_E^F(\Theta)\) and \(\sigma_2 = \mathrm{Ind}_L^F(\chi)\) where \(\Theta = \theta \Delta_\theta^{-1}\) and \(\chi = \eta \Delta_\eta^{-1}\).

\textbf{Case 3.a} Suppose \(L \ne E\). Let \(K = EL\), so \(\mathrm{Gal}(K/F) = \{e, \tau_1, \tau_2, \tau_2\tau_1\}\) where \(\tau_1\) fixes \(E\) and \(\tau_2\) fixes \(L\). Let \(T\) be the intermediate field fixed by \(\tau_1\tau_2\). Then
\begin{equation*}
\begin{split}
&\epsilon(\pi_\theta \times \pi_\eta, \psi) = \epsilon\left( \sigma_1 \otimes \mathrm{Ind}_L^F(\chi), \psi \right) \\
&= \epsilon\left( \sigma_1|_L \otimes \chi, \psi_L \right) \omega_{L/F}(-1) = \epsilon\left( \Theta \circ N_{K/E} \otimes \chi \circ N_{K/L}, \psi_K \right) \omega_{L/F}(-1).
\end{split}
\end{equation*}
We claim that \(\Theta \circ N_{K/E} \otimes \chi \circ N_{K/L}\) is trivial on \(T^\times\). For \(t \in T^\times\), note that \(\tau_1|_T\) must be a nontrivial element of \(\mathrm{Gal}(T/F)\) (otherwise \(\tau_1\) acts trivially on both \(E\) and \(T\), implying \(K = ET\) which is a contradiction). By definition, for \(t \in T^\times\) we have \(N_{K/E}(t) = t \cdot \tau_1(t) = t \cdot (\tau_1|_T)(t) = N_{T/F}(t) \in F\). It suffices to verify that \(\omega_{E/F}(N_{T/F}(t)) = \omega_{L/F}(N_{T/F}(t))\). By \cite[Corollary 1.2]{Mil20}, \(N_{L/F}(T^\times) \cap N_{E/F}(E^\times) = N_{K/F}(K^\times) = N_{L/F}(T^\times) \cap N_{L/F}(L^\times)\), so the claim holds. By this claim and Theorem \ref{epsilon}, we have 
\[
\epsilon\left( \Theta \circ N_{K/E} \otimes \chi \circ N_{K/L}, \psi_K \right) = c \Theta(N_{K/E}(\Delta)) \chi(N_{K/L}(\Delta))
\]
where \(c\) is a constant independent of \(\eta, \theta\), and \(\Delta \in K^\times\) satisfies \(\tau_2\tau_1(\Delta) = -\Delta\). Let \(E = F(\sqrt{a})\) for \(a \in F\); we may choose \(\Delta = \sqrt{a}\), which is \(\tau_1\)-invariant and \(\tau_2\)-anti-invariant. Then \(\Theta(N_{K/E}(\Delta)) = \omega_{E/F}(a)\) and \(\chi(N_{K/L}(\Delta)) = \omega_{L/F}(-a)\), both independent of \(\Theta\) and \(\chi\). Thus the statement holds.

\textbf{Case 3.b} Suppose \(L = E\). Let \(\tau \in \mathrm{Gal}(E/F)\) be the nontrivial element. Then
\begin{equation*}
\begin{split}
\epsilon(\pi_1 \times \pi_2) &= \epsilon\left( \mathrm{Ind}_E^F(\chi) \otimes \mathrm{Ind}_E^F(\Theta) \right) = \epsilon\left( (\mathrm{Ind}_E^F(\chi))|_F \otimes \Theta, \psi_E \right) \\
&= \epsilon\left( (\chi \oplus \chi^\tau) \otimes \Theta, \psi_E \right) = \epsilon\left( \chi \otimes \Theta, \psi_E \right) \cdot \epsilon\left( \chi^\tau \otimes \Theta, \psi_E \right)
\end{split}
\end{equation*}
It is straightforward to check that \(\chi \otimes \Theta\) and \(\chi^\tau \otimes \Theta\) are trivial on \(F^\times\), so we may apply Theorem \ref{epsilon}. Let \(\Delta \in E^\times\) be an element with trace zero in \(F\); then \(\epsilon\left( \chi \otimes \Theta, \psi_E \right) \cdot \epsilon\left( \chi^\tau \otimes \Theta, \psi_E \right) = c \chi(\Delta) \Theta(\Delta) \chi(\tau(\Delta)) \Theta(\tau(\Delta)) = c^2 \Theta(\Delta^2) \chi(-\Delta^2) = c^2 \omega_{E/F}(-1)\), where \(c\) is independent of \(\eta\) and \(\theta\). Hence the statement holds.
\end{proof}

\section{Bounds for local integrals}\label{section3}

In this section, we summarize some basic results on Whittaker models following \cite{HMN23} and \cite{Hu17}. We then employ the \(p\)-adic stationary phase method to obtain a sharp bound for the local integral \(\mathcal{P}\). This section is purely local in nature. For conciseness, we shall omit the subscript \(p\) from our notation.

\subsection{Basics of Whittaker functions}

Suppose $p\ne 2$ and $\pi$ is an irreducible representation of $\mathrm{GL}_2(F)$ with trivial central character. Let $W_{\pi}$ be the Whittaker function of new form normalized by $W_{\pi}(1) =1$. We set $\mathfrak c = \mathfrak c(\pi)$ be the log-conductor.

\begin{lemma} For every positive integer $\mathfrak c$,
\begin{equation*}
\mathrm {GL}_2 = \bigcup _{0\le i\le \mathfrak c } B \left (
\begin{array}{cc}
    1 & 0 \\
    \varpi ^i  & 1
\end{array} \right ) K_0 (\mathfrak c).
\end{equation*}
Here $B$ is the Borel subgroup.
\end{lemma}

\begin{definition}\label{definitionW^k}
We define
\[
    W_{\pi}^{(k)}(x) =W_{\pi} \left (\left (
\begin{array}{cc}
    x & 0 \\
    0 & 1
\end{array}
\right)
\left (
\begin{array}{cc}
    1 & 0 \\
    \varpi^k & 1
\end{array}
\right)
\right ).
\]
\end{definition}

\begin{lemma}
    For $k\le \mathfrak c$, $W_{\pi} ^{(k)}(x)$ is $U(\mathfrak c-k)$-invariant, where we set $U(k) = 1+\mathfrak p^k$.
\end{lemma}

\begin{proof}
Let $1+u \in U(\mathfrak c-k)$, note that  
\[ 
\left (
\begin{array}{cc}
    x(1+u) & 0 \\
    0 & 1
\end{array}
\right)
\left (
\begin{array}{cc}
    1 & 0 \\
    \varpi^k & 1
\end{array}
\right)
=
\left (
\begin{array}{cc}
    x & 0 \\
    0 & 1
\end{array}
\right)
\left (
\begin{array}{cc}
    1 & 0 \\
    \varpi^k & 1
\end{array}
\right)
\left (
\begin{array}{cc}
    1+u & 0 \\
    -u\varpi^k & 1
\end{array}
\right),
\]
and $\left (\begin{array}{cc}
    1+u & 0 \\
    u\varpi ^k & 1
\end{array}
\right)\in K_1(\mathfrak p^{\mathfrak c})$. Thus the statement holds.
\end{proof}

\begin{lemma}[\cite{HMN23}, proof of Lemma 6.1] \label{Atkin}
Let \(i < \mathfrak{c}/2\); then we have
\begin{equation*}
W^{(i)}_\pi (x) = a_i W^{(\mathfrak{c} - i)}_\pi (- \varpi^{\mathfrak{c} - 2i}x) \psi(\varpi^{-i}x).
\end{equation*}
Here \(a_i = \pm 1\) is a constant.
\end{lemma}
Thus we only concern ourselves with the case where \(k \ge \mathfrak{c}/2\).

We now give the explicit expression of the new form. Let \(\pi\) be an irreducible representation of \(\mathrm{GL}_2(F)\) associated to the pair \((E, \theta)\) via compact induction or parabolic induction, where \(E\) is an étale quadratic algebra over \(F\) and \(\theta\) is a character of \(E^\times\) such that \(\theta|_{F^\times}\) is trivial.

If \(E\) is a field, then \(\pi\) is supercuspidal. We denote by \(\mathfrak{o}_E\) the ring of integers, \(\varpi_E\) the uniformizer, \(\nu_E\) the discrete valuation, and \(e_E\) the ramification index of \(E\).

If \(E\) is a split quadratic algebra of \(F\), then \(\pi\) is a principal series representation. We identify \(E\) with \(F \times F\). Define \(\mathfrak{o}_E = \mathfrak{o}_F \times \mathfrak{o}_F\), \(\mathfrak{o}_E^\times = \mathfrak{o}_F^\times \times \mathfrak{o}_F^\times\), and \(\nu_E(x) = i\) when \(x \in \varpi^i \mathfrak{o}_E^\times\). (In particular, \(\nu_E\) is not defined for all \(x \in E\) in this case.)

We set \(N: E^\times \to F^\times\) and \(T: E \to F\) to be the norm and trace operators, respectively. 

Recall from \cite[Lemma 5.5, Lemma 5.7]{HS20} the following result, which is a reformulation of \cite[Lemma 3.1]{Ass21} and holds actually for all dihedral supercuspidal representations.

\begin{lemma}\label{Whittaker function}
As a function in $x$, $W^{(k)}(x)$ is supported on $\mathfrak o^\times $ if $\mathfrak c\ge k>\mathfrak c/2$, and supported on $\mathfrak o$ if $k = \mathfrak c/2$. For $\mathfrak c \ge k\ge \mathfrak c/2$, $W^{(k)} (x)$ consists only of level $\mathfrak c -k$ components (in the sense of Mellin transform),  except when $k = \mathfrak c - 1$
where it consists of level $\le 1$ components. And if  $x\in \mathfrak o^\times$  we have:
\begin{equation*} 
W^{(k)}(x) = C_0^{-1} \int _{\nu _E (u )= -c(\theta ) -e_E +1} \theta ^{-1}(u) \psi \left (T(u)- \frac 1 x \varpi ^k N (u)  \right ) d^\times u. 
\end{equation*}
where 
\begin{equation*}
C_0 = \int _{\nu _E (u) = - c(\theta ) - e_E +1 } \theta ^{-1} (u) \psi _E (u) d^\times u.
\end{equation*}
The normalization by $C_0$ guarantees that $W^{(k)} (1) =1$. By the estimate of Guass sum we have 
\begin{equation*}
|C_0| = |  \mathfrak o_E /  \mathfrak p_E |^{-c(\theta )/2}\asymp p^{c(\theta )/e_E} .
\end{equation*}
\end{lemma}


\subsection{\texorpdfstring{The $p$-adic stationary phase method}{The p-adic stationary phase method}}
\begin{definition}
Let $D$ be a $\mathfrak p^k$-invariant subset of $F$ and $f: D\to \mathbb C$. Then $f$ is called $k$-additive or $\mathfrak p^k$-additive if for any $x\in D$, there exists a  parameter $a(x)$ such that for any $\delta \in \mathfrak p^k$ , we have 
\begin{equation*}
 f(x+\delta ) = f(x) \psi (a(x) \delta). 
\end{equation*}
And $a(x)$ is called additive parameter with respect to $\psi$. If $f $ is $k$-additive then we define
\[
    D(f,k) = \{ x\in D/\mathfrak p^k; \psi (a(x)\delta ) =1, \forall \delta \in \mathfrak p^k\} = \{ x\in D/\mathfrak p^k; a(x)\in \mathfrak p^{-k+c(\psi )}\}.
\]
\end{definition}

\begin{remark}
The $k$-additivity and $D(f,k)$ does not depend on the choice of additive character $\psi$ but the additive parameter $a(x)$ depends on $\psi$.
\end{remark}

If $f$ is $k$ -additive, then we can compute the integral
\begin{equation*}
\int _D f(x) dx = \frac 1 {p^k} \sum _{x\in D(f,k)} f(x).
\end{equation*}

Now we give some examples.
\begin{lemma}\label {character}
Let $\theta$ be a character. There exists $\alpha _\theta \in F$ such that $\nu (\alpha _\theta ) = -c(\theta ) +c(\psi)$ and for any   $u \in \mathfrak p^{\lceil c(\theta)/2 \rceil}$,
 \[
 \theta(1+u) = \psi (\alpha _ \theta u).
 \]
 Then for any $x\in \varpi ^k \mathfrak o^\times$ and $u\in \mathfrak p^{\lceil c(\theta)/2 \rceil+k}$ we can see
 \[
 \theta (x+u) = \theta (x) \theta (1+u/x) = \theta (x) \psi (\alpha _\theta u/x)
 \]
 Thus $\theta |_{\varpi ^k \mathfrak o^\times}$  is $\lceil c(\theta)/2 \rceil+k$-additive with $a(x) = \alpha _\theta /x$.
\end{lemma}

\begin{remark}
However, one can see that $\theta$ is  not $k$-additive for any $k$.
\end{remark}

\begin{definition}
Let $D$ be a $\mathfrak p^k $-invariant subset. We say a function $f: D\to F$ is $k$-differentiable   if for any $x\in D$, there exists $f'(x) \in F$ such that for any $\delta  \in \mathfrak p^k$ we have 
\begin{equation*}
f(x+\delta ) \equiv  f(x) + \delta  f'(x) \mod {\mathfrak o},
\end{equation*}
or equivalently  we can write 
\begin{equation*}
f(x+\delta ) =  f(x) + \delta  f'(x) +O(\mathfrak o).
\end{equation*}
And $f'$ is called the differential of $f$.
\end{definition}

\begin{lemma}
Suppose $f$ is $k$-differentiable with differential $f'$. Then $\psi (f)$ is $k$-additive with $a(x) = f'(x)$.
\end{lemma}

Thus if $f$ is $k$-differentiable, we have 
\begin{equation*}
\int _D \psi (f(x)) dx=\sum _{x\in D(\psi (f),k) } \psi (f(x))/p^k.
\end{equation*}
Now we study a simple case which can be viewed as a kind of Guass sum. Define 
\[
   I(a,b) = \int _{\mathfrak o^\times } \psi (ax^2+bx) dx.  
\]
\begin{lemma} \label{4} Let $k = -\nu (a)\ge 1$.
If $k = 1$, then $I(a,b)$ is nonvanishing only if $\nu (b) \ge -1$.
If $k \ge 2$. Let $a = a_0\varpi ^{-k}$ where $a_0 \in \mathfrak o^\times$.  Denote  $\left ( \frac {a_0} {\mathfrak p} \right )$ by the Legendre symble on $\mathfrak o^\times /\mathfrak p$. Then the integral 
\[
I(a,b)= \begin{dcases}
      \frac {G(a)} { p^{k/2} }  \psi \left ( -\frac {b^2} {4a}\right ), \quad & \nu(a) = \nu(b)\\
     0. \quad & \text{otherwise}
\end{dcases}
\]
where
\[ 
G(a) = \begin{dcases}
    \frac {\sum _{x_0\in \mathfrak o/\mathfrak p} \psi \left ( \frac { x_0^2} \varpi \right )}{\sqrt p}\left ( \frac {a_0} {\mathfrak p}\right ), \quad & k \ \text{odd} \\
    1. \quad & k \ \text{even}
\end{dcases}
\]
\end{lemma}
\begin{proof}
Set $f(x) = ax^2+bx$, then $f(x+\delta ) = f(x) +\delta (2a x +b) +a \delta ^2$. 
Thus we can see $f$ is $\lceil k/2 \rceil$-differentiable  with $f'(x) = 2ax +b$. Thus we have 
\begin{equation*}
I(a,b) = p^{-\lceil k/2 \rceil}\sum _{x\in \mathfrak o^\times / \mathfrak p^{\lceil k/2 \rceil}; 2ax +b \in \mathfrak p^{-\lceil k/2 \rceil}} \psi (ax^2 +bx)  . 
\end{equation*}
If $k=1$ and  $I(a,b)$ is nonvanishing, then $b\in \mathfrak p^{-1}$. If $k \ge 2$, then $\nu (a) = \nu (b)$, $x = -b/2a  \mod \mathfrak p^{\lfloor k/2 \rfloor}$ and 
\begin{equation*}
\begin{split}
I (a,b) &= p^{-\lceil k/2 \rceil} \sum _{x_0\in \mathfrak p^{\lfloor k/2 \rfloor }/\mathfrak p^{\lceil k/2 \rceil}} \psi \left (a\left (x_0-\frac b {2a}\right )^2+b\left (x_0-\frac b {2a}\right )\right )\\
& = \frac {\psi \left ( -\frac {b^2} {4a}\right )} {p^ {\lceil k/2 \rceil}} \sum_{x_0\in \mathfrak p^{\lfloor k/2 \rfloor }/\mathfrak p^{\lceil k/2 \rceil}} \psi (ax_0^2).
\end{split}
\end{equation*}
If $k$ is even, the statement is clear. If $k$ is odd, we set $k = 2n+1$ and $a= a_0\varpi ^{-2n-1}$. Then it is easy to check that the last summation equals  
\[
    \sum _{x_0\in \mathfrak o/\mathfrak p} \psi \left ( \frac {a _0 x_0^2} \varpi \right )= G(a)\sqrt p.
\] 
Thus the statement holds.
\end{proof}

Now we apply $p$-adic stationary phase to  whittaker function. 
\begin{corollary}\label{degenerate}
    For $\mathfrak c$ sufficiently large and $x\in \mathfrak o$ we have 
    \begin{equation*}
    W_\pi ^{\mathfrak c -1} (x) \ll 1.
    \end{equation*}
\end{corollary}

\begin{proof}
   We prove the principle series case here and other cases are similar. Let $E= F\times F$ and $\theta = \mu^{-1} \otimes \mu$. By Lemma \ref{4} we have 
\begin{equation*}
\begin{split}
W^{(\mathfrak c-1)}(x) &= C_0^{-1}\int_{\mathfrak o^\times \times \mathfrak o^\times } \psi\left (\frac {y+z} {\varpi ^{\mathfrak c/2}}-\frac {y z} {x\varpi }   \right ) \mu (yz^{-1}) dy dz \\
&= C_0^{-1}\int_{\mathfrak o^\times \times \mathfrak o^\times } \psi\left (\frac {(y+1)z} {\varpi ^{\mathfrak c/2}}-\frac {y z^2} {x\varpi }   \right ) \mu (y) dz dy
\end{split}
\end{equation*}
Then the inner integral is nonvanishing only if $\nu (y+1)\ge \mathfrak c/2 -1$. Thus 
\begin{equation*}
W^{(\mathfrak c-1)} \ll p^{\mathfrak c/2} \mathrm {Vol} (\mathfrak p^{\mathfrak c/2-1}-1,dx)\asymp 1.
\end{equation*}
Thus the statement holds.
\end{proof}

\begin{definition} Let $E$ be an etale quadratic algebra over $F$ and $\theta$ be a character of $E^\times$. Let $k$ be a positive integer. If $E$ is a field, Let  $D\in E$  such that $E = F (\sqrt D)$, $\nu _E(\sqrt D)=e_E -1$,
Then  let $W[H, k,E,\theta ] (x)$ denote the integral :
\begin{equation*}
\int _{H} \psi \left (\frac {xf_E(y)} {\varpi ^k} \right ) \theta (\mathcal L_E(y)) dy
\end{equation*}
where 
\begin{equation*}
\begin{split}
    f_E(y)&=\begin{dcases}
        \frac {y^2} {4y-4}, & \quad E \text{\ split}\\
        \frac {y^2} {y^2-D}, & \quad E \text {\ nonsplit}
    \end{dcases}\\
    \mathcal L_E(y) &= \begin{cases}
        (y-1,1), &\quad E \text{\ split}\\
        y -\sqrt D, & \quad E \text {\ nonsplit} 
    \end{cases}\\
\end{split}
\end{equation*}
and $H \subseteq F$ is compact such that $f_E$ is well-defined.
\end{definition}

\begin{lemma}\label{another} Suppose $\pi$ is associated to $(E,\theta)$ via parabolic induction or compact induction. Let $\mathfrak c = \mathfrak c(\pi)$. If $E$ is not a ramified extension, 
for $\mathfrak c-1>k> \mathfrak c/2$ and $x \in \mathfrak o^\times$, we have 
\begin{equation*}
 W^{(k)}_\pi (x) = C_0^{-1} p^{k/2-\mathfrak c/2} G(x\varpi ^{\mathfrak c-k})W[\varpi ^{k-c(\theta )/e_E}\mathfrak o^\times, k, E,\theta ](x),
\end{equation*}
where $G(x)$ is defined in the Lemma \ref{4}.

\end{lemma}

\begin{proof} We prove this lemma case by case.

\textbf{Case 1. $E$ split.} Then $E= F\times F$ and $\theta = \mu^{-1} \otimes \mu$. By Lemma \ref{Whittaker function} and Lemma \ref{4}, for $\mathfrak c-1>k \ge  \mathfrak c/2$ and $\nu(x)=0$  we have:   
\begin{equation}\label{Esplit}
\begin{split}
W^{(k)}(x) &= C_0^{-1} \int_{\mathfrak o^\times \times \mathfrak o^\times } \psi\left (\frac {y+z} {\varpi ^{\mathfrak c/2}}-\frac {y z} {x\varpi ^{\mathfrak c-k}}   \right ) \mu (yz^{-1}) dy dz \\
&= C_0^{-1}\int_{\mathfrak o^\times \times \mathfrak o^\times } \psi\left (\frac {(y+1)z} {\varpi ^{\mathfrak c/2}}-\frac {y z^2} {x\varpi ^{\mathfrak c-k}}   \right ) \mu (y) dz dy\\
&= C_0^{-1} p^{k/2-\mathfrak c/2}\int _{ \nu (y+1) = k- \mathfrak c /2,\  y\in \mathfrak o^\times } G\left (-\frac {y} {x\varpi ^{\mathfrak c-k}}\right )\psi \left ( \frac {(y+1)^2x} {4 \varpi ^k y}\right ) \mu (y) dy\\
&= C_0^{-1} p^{k/2-\mathfrak c/2}G(x\varpi ^{\mathfrak c-k}) \int _{ \varpi ^{k- \mathfrak c /2}\mathfrak o^\times } \psi \left ( \frac {y^2x} {4 \varpi ^k (y-1)}\right ) \theta ^{-1} (y-1) dy.
\end{split}
\end{equation}
Here we explain the last equation. It is easy to see that $G(a) = G(1/a)$. Then $-y \equiv 1  \mod \mathfrak p$ and
\[G\left (-\frac {y} {x\varpi ^{\mathfrak c-k}}\right )=G\left (\frac {1} {x\varpi ^{\mathfrak c-k}}\right )=G({x\varpi ^{\mathfrak c-k}}).\] 

\textbf{Case 2. $E$ unramified.} Let $E=F(\sqrt {D})$, and  $\varpi_E =\varpi $. Then
\[ 
W^{(k)}(x)  = \frac 1 {C_0} \int _{\mathfrak o_E^\times } \theta ^{-1}(u) \psi \left (\varpi ^{-c(\theta )}T( u)- \frac  {\varpi  ^k} {x\varpi ^{2c (\theta) }} N (u)  \right ) d^\times u 
\]
We can see $\mathfrak o_E ^\times = \left \{ y+z\sqrt D ; y\in \mathfrak o^\times , z\in \mathfrak o \right \}\sqcup \left \{ y+z\sqrt D ; y\in \mathfrak p , z\in \mathfrak o^\times  \right \}  $. Thus we have
\begin{equation*}
\begin{split}
\int _{a\in \mathfrak o_E^\times } f(a) d^\times a &= \frac {p-1} {p+1} \left ( \int _{ \mathfrak o^\times \times \mathfrak o } f(y+z\sqrt {D} )dydz + \int _{\mathfrak p \times \mathfrak o^\times} f(y+z\sqrt D)dy dz  \right )\\
&= \frac {p-1} {p+1} \left ( \int _{ \mathfrak o^\times \times \mathfrak o } f(y(1+z\sqrt {D}) )dydz + \int _{\mathfrak p \times \mathfrak o^\times} f(z(y+\sqrt D))dy dz  \right )
\end{split}
\end{equation*}
Note that $\theta$ is trivial on $F$. We can see $W^{(k)} (x) = W_1 +W_2$ , where  
\begin{equation*}
\begin{split}
 W_1 &= C_0^{-1} \int_{ \mathfrak o^\times \times  \mathfrak o} \psi \left (   \frac {2y} {\varpi ^{c(\theta )}}-\frac {y^2(1-Dz^2) } {x\varpi ^{2c(\theta)-k}  }  \right )  \theta^{-1} (1+z \sqrt D) dy dz \\
 W_2 &=C_0^{-1}\int  _{\mathfrak p \times \mathfrak o^\times }\psi  \left ( \frac {2yz } {\varpi ^{c(\theta )}} -\frac {(y^2 -D)z^2} { x\varpi ^{2c(\theta)-k} } \right ) \theta^{-1} (y+\sqrt D) dydz
\end{split}
\end{equation*}
By Lemma \ref {4}, the inner integral of $y$ of $W_1$  is nonzero only if  $-c(\theta )\ge k-2c(\theta )$. So $W_1$ vanishes when   $k > c(\theta)$. And $W_2$ equals
\begin{equation*}
\begin{split}
 &C_0^{-1}p^{k/2-c(\theta)/2}\int  _{\varpi^{k-c(\theta )} \mathfrak o^\times  } G\left (-\frac {y^2-D} {x \varpi ^{2c(\theta ) - k }} \right )\psi  \left ( \frac {  y^2x } {\varpi ^k (y^2-D)} \right ) \theta^{-1} (y+\sqrt D) dy\\
&= -C_0^{-1}p^{k/2-c(\theta)}G(x\varpi ^{2c(\theta ) - k }) \int _{\varpi^{k-c(\theta )} \mathfrak o^\times } \psi \left ( \frac {  y^2x } {\varpi ^k (y^2-D)} \right  ) \theta ^{-1} (y+\sqrt D) dy. 
\end{split}
\end{equation*}
\end{proof}

\begin{lemma}\label{ramifiedextension}
If $E$ is a ramified extension, 
for $\mathfrak c-1>k> \mathfrak c/2$ and $x \in \mathfrak o^\times$, we have 
\begin{equation*}
 W^{(k)}_\pi (x) = C_0^{-1} p^{k/2-c(\theta )/2} G(xD^{c(\theta +1)} \varpi ^k)W[\varpi ^{k-c(\theta )/e_E}\mathfrak o^\times, k, E,\theta ](x).
\end{equation*}
\end{lemma}
\begin{proof} In this case $c(\theta )$ is even and $\mathfrak c =c(\theta )+1$. And we set $\varpi _E = \sqrt D$. Then
\[ 
W^{(k)} (x) =C_0^{-1} \int _{\mathfrak  o_E^\times } \theta ^{-1} (u) \psi \left (\frac 1 {D^{c(\theta )/2}} T\left (\frac u {\sqrt D} \right )+\frac  {\varpi^i} {xD^{c(\theta )+1}} N(u)\right ) d^\times u
\]
Then $\mathfrak o_E^\times = \{ y+z \sqrt D ; y\in \mathfrak o^\times , z\in \mathfrak o\}$, and we have 
\begin{equation}\label{ramified}
\begin{split}
&W^{(i)}(x) = C_0^{-1}\int _{ \mathfrak o^\times  \times \mathfrak o} \psi \left (\frac {\varpi ^k(y^2-Dz^2)} {xD^{c(\theta)+1}} +  \frac {2z} {D^{c(\theta)/2}}\right )\theta^{-1} (y+z\sqrt D)dydz\\
&= C_0^{-1} \int _{\mathfrak o^\times \times  \mathfrak o} \psi \left (\frac {\varpi ^k(1-Dz^2)y^2} {xD^{c(\theta )+1}} +  \frac {2yz} {D^{c(\theta)/2}}\right )\theta^{-1} (1+z\sqrt D)dydz\\
&= C_0^{-1}p^{k/2-c(\theta )/2 -1/2 } \times \\
& \quad \int _{ \varpi ^{k-c(\theta )/2-1} \mathfrak o^\times  }G\left (\frac {\varpi ^k(1-Dz^2)} {xD^{c(\theta )+1}} \right ) \psi \left (-\frac {Dxz^2 } {\varpi ^k(1-Dz^2)}\right )\theta^{-1} (1+z\sqrt D)dz\\
&= C_0^{-1}p^{k/2-c(\theta )/2 -1/2 }G(xD^{c(\theta )+1} \varpi ^k)\int _{ \mathfrak p^{k-c(\theta )/2} \mathfrak o^\times  } \psi \left (\frac {xz^2 } {\varpi ^k(z^2-D)}\right )\theta^{-1} (z+\sqrt D)dz.\\
\end{split}
\end{equation}
\end{proof}

\begin{lemma}\label{split} 
Let $\pi , E, \theta ,\mathfrak c$ be same as Lemma \ref{another}.
If $E$ is split and $x\in \mathfrak o^\times $, then 
\begin{equation*}
W^{\mathfrak c/2} (x)= C_0^{-1}\left ( \frac {-1} {\mathfrak p}\right )G(x\varpi ^{\mathfrak c-k})W \left [\mathfrak o^\times \backslash (1+\mathfrak p), \mathfrak c/2 , E, \theta \times \left ( \frac {\cdot} {\mathfrak p} \right ) \right ].
\end{equation*}
\end{lemma}

\begin{proof}
 Recall the second to last line of (\ref{Esplit}), we get that $y\in \mathfrak o^\times \backslash (\mathfrak p-1)$. Then $y+1 \in \mathfrak o^\times \backslash (\mathfrak p+1)$ and 
\[
G\left (-\frac {y} {x\varpi ^{\mathfrak c-k}}\right )=\left ( \frac {-1} {\mathfrak p}\right )\left (\frac {y} {\mathfrak p} \right ) G(x\varpi ^{\mathfrak c-k}).
\]
Thus the statement holds.
\end{proof}

\begin{lemma}\label{unramifedfieldextension}
Let $\pi , E, \theta ,\mathfrak c$ be same as in Lemma \ref{another}.
If $E$ is nonsplit, $\mathfrak c$ is even and $x\in \mathfrak o^\times $, then 
\begin{equation*}
\begin{split}
&W^{\mathfrak c/2} (x)= C_0^{-1} G({x\varpi ^{2c(\theta)-k}  })\times \\
&\left  (W \left [\mathfrak o^\times , \mathfrak c/2 , E, \theta \times \left ( \frac {N(\cdot)} {\mathfrak p} \right ) \right ]+ \sum _{i\ge 1} p^{-2i}  W[\varpi ^{-i} \mathfrak o^\times  , c(\theta ) , E,\theta ]\right ).
\end{split}
\end{equation*}
\end{lemma}
\begin{proof}
Then $E$ must be unramified, and recall that we set $W^{(\mathfrak c/2)}(x) =W_1+W_2$.
$W_2$ is  $0$ if $k = c(\theta )$ by Lemma \ref {4} and $y\in \mathfrak p$. Then we only need to consider $W_1$ which equals
\begin{equation}\label{supercuspidal}
\begin{split}
    &C_0^{-1} p^{k/2-c(\theta)}\int _{ \mathfrak o} G\left (-\frac {1-Dz^2} {x\varpi ^{2c(\theta)-k}  } \right )\psi \left ( \frac x {(1-Dz^2) \varpi ^{c(\theta)}}\right ) \theta ^{-1} (1+z\sqrt D) dz\\
    & = \sum _{i = 0}C_0^{-1} p^{k/2-c(\theta)} \int _{ \varpi ^i \mathfrak o^\times } G\left (-\frac {1-Dz^2} {x\varpi ^{2c(\theta)-k}  } \right )\psi \left ( \frac x {(1-Dz^2) \varpi ^{c(\theta)}}\right ) \theta ^{-1} (1+z\sqrt D) dz \\
    &=:  \sum _{i=0} W[i].
\end{split}
\end{equation}
For  $W[0]$, we have 
\begin{equation}
\begin{split}
W[0] &=C_0^{-1} G({x\varpi ^{2c(\theta)-k}  }) \int _{ \mathfrak o^\times } \left (\frac {Dz^2-1} {\mathfrak p}  \right )\psi \left ( \frac x {(1-Dz^2) \varpi^{c(\theta)}}\right ) \theta ^{-1} (1+z\sqrt D) dz \\
& =C_0^{-1}G({x\varpi ^{2c(\theta)-k}  }) \int _{ \mathfrak o^\times } \left (\frac {D-z^2} {\mathfrak p}  \right )\psi \left ( \frac {xz^2} {(z^2-D) \varpi^{c(\theta)}}\right ) \theta ^{-1} (z+\sqrt D) dz \\
&= C_0^{-1} G({x\varpi ^{2c(\theta)-k}  })W \left [\mathfrak o^\times , \mathfrak c/2 , E, \theta \times \left ( \frac {N(\cdot)} {\mathfrak p} \right ) \right ]
\end{split}
\end{equation}
For $i\ge 1$, we have 
\begin{equation*}
\begin{split}
W[i] &=C_0^{-1} G({x\varpi ^{2c(\theta)-k}  }) \int _{ \varpi ^i \mathfrak o^\times } \psi \left ( \frac x {(1-Dz^2) \varpi^{c(\theta)}}\right ) \theta ^{-1} (1+z\sqrt D) dz \\
& =C_0^{-1} p^{-2i} G({x\varpi ^{2c(\theta)-k}  }) \int _{ \varpi ^{-i} \mathfrak o^\times } \psi \left ( \frac {xz^2} {(z^2-D) \varpi^{c(\theta)}}\right ) \theta ^{-1} (z+\sqrt D) dz \\
& = C_0^{-1} p^{-2i} G({x\varpi ^{2c(\theta)-k}  }) W[\varpi ^{-i} \mathfrak o^\times  , c(\theta ) , E,\theta ].
\end{split}
\end{equation*}
\end{proof}

We give a lemma which will be used next section. 

\begin{lemma}\label{Taylor}
Let $k \ge 1$. The rational function $f_E(x)/\varpi ^k $ defined in the Lemma \ref{another} is   $\lceil k/2 \rceil$-differentiable  for $x\in  F$ except when $E$ is split and $x\in 1+\mathfrak p$. 
\end{lemma}

\begin{proof}
By direct calculation we have
\[
    f_E(x+\delta )  = f_E(x) + \delta f_E'(x) + \delta^2 h(x,\delta ),
\]
where  
\begin{equation*}
\begin{split}
f_E' (x) &=  \begin{dcases}
    \frac{y^2 - 2y}{4(y - 1)^2}, & E \text {\ split}\\
    -\frac{2Dy}{(y^2 - D)^2}, & E \text {\ nonsplit}
\end{dcases}\\
h(x,\delta ) &= \begin{dcases}
\frac 1 {(x-1)^2(x+\delta -1)}, &  E \text {\ split}\\
\frac {3Dx^2+2Dx \delta +D^2} {(x^2-D)^2((x+\delta)^2-D)}. & E \text {\ nonsplit}
\end{dcases}
\end{split}
\end{equation*}
Now if $x\notin 1+\mathfrak p$, $\delta\in \mathfrak p^{\lceil k/2 \rceil}$, then it is easy to see $\delta  ^2 h(x,\delta ) \in \mathfrak o$. Thus the statement holds.
\end{proof}

\subsection {Main local result}
In the rest of this section, we assume that $\pi_1$ (resp. $\pi _2$) is associated with $(E,\theta )$ (resp. $(L,\eta )$) via parabolic induction or compact induction. Let $\mathfrak c_1 = \mathfrak c(\pi _1)$ (resp $\mathfrak c_2 = \mathfrak c(\pi _2)$).  
Our aim is to find a bound of  
\begin{equation}\label{mathcalP}
\mathcal P:= \int _{x\in \mathfrak o^\times } W^{(i)} _{\pi_1} (x) W^{(j)} _{\pi _2} (x) \psi (x/ \varpi ^l) dx.
\end{equation} 
Fix $k$ and let $x\in \mathfrak o^\times$, note that $G(x\varpi ^k )$ only depends on $\left (\frac x {\mathfrak p} \right ) $. Then the problem is reduced to consider the integral 
\begin{equation*}
\mathcal P(x_0)=:\int _{x\in \mathfrak p} W[H_1, i,E,\tilde \theta](x_0+x) W[H_2,j,L,\tilde \eta] (x_0+x)\psi \left (\frac {x_0+x} {\varpi ^l} \right ) dx
\end{equation*}
where $\tilde \theta = \theta$ or $\theta \times \left ( \frac {N(\cdot) } {\mathfrak p} \right )$ up to $H_1$, $i$ , $E$.

\begin{proposition}\label{mainlocal} Suppose $\mathfrak c_1\ge  i> \mathfrak c_1/2$, $\mathfrak c_2\ge j \ge \mathfrak c_2/2$ and $\mathfrak c_2-j > \mathfrak c_1-i$. 
Let $H_1=  \varpi ^{ i- c(\theta )/e_E}\mathfrak o^\times $.  For $H_2$, we consider following cases:\\
$(\romannumeral1)$  $j> \mathfrak c_2/2$, and $H_2 = \varpi ^{j -c(\eta )/ e_L} \mathfrak o^\times$,\\
$(\romannumeral2)$  $j = \mathfrak c_2/2$, $L$ is split and  $ H_2 =  \mathfrak o^\times \backslash (1+\mathfrak p\cup 2+\mathfrak p)$,\\
$(\romannumeral3)$ $j = \mathfrak c_2/2$, $L$ is split and  $H_2=2+\varpi ^k \mathfrak o^\times $ for some positive $k$,\\
$(\romannumeral4)$ $j = \mathfrak c_2/2$, $L$ is a field extension  and  $H_2= \mathfrak o^\times $,\\
$(\romannumeral5)$ $j = \mathfrak c_2/2$, $L$ is a field extension  and  $H_2= \varpi ^{-k}\mathfrak o^\times $ for some positive $k$.\\
Then $\mathcal P(x_0)$ is nonzero only if $(\romannumeral1)$ $(\romannumeral2)$ $(\romannumeral4)$ and  $l = \mathfrak c_2-j$ and we have 
\begin{equation*}
\mathcal P(x_0)\ll _p p^{- i/2  -\mathfrak c_2/2}.
\end{equation*}
\end{proposition}

For clarity, we first summarize the notation conventions. We set $f(y)= f_E(y)$, $g(z) = f_L(z)$, $\mathcal L = \mathcal L_E$ and $\mathcal M = \mathcal L_L$.

\begin{remark}\label{remark}
Let \( L \) be a field extension. Then, \(\nu(g'(z)) = \nu(z) + \nu(D_2)\) holds if and only if \(\nu(z)\) is non-negative. Suppose \( L \) is split and \(\nu(z)\) is non-negative; from prior analysis, \(\nu(g'(z)) = \nu(z) +\nu(D_2)\) is valid if and only if \(\nu(z)\) is positive or \( z \in \mathfrak o^\times \setminus (1+\mathfrak p \cup 2+\mathfrak p)\).
\end{remark}

\begin{corollary}\label{Boundsforlocal}
Let $i,j$ be  variables within the same constraints as specified in Proposition \ref{mainlocal}.  Consequently, the $\mathcal P$, defined in equation (\ref{mathcalP}), becomes non-zero when $l = \mathfrak c_2-j$. In this context, we can assert that 
\begin{equation}
\mathcal P\ll _p p^{j/2  -\mathfrak c_2/2}.
\end{equation}
\end{corollary}

\begin{proof}(Proposition \ref{mainlocal} implies Corollary \ref{Boundsforlocal}) First of all, by Lemma \ref{Whittaker function}, we have $C_0 \asymp p^{\mathfrak c/2}$. Then for $j > \mathfrak c_2/2$, by Lemma \ref{another}, Lemma \ref{ramifiedextension}, we have that 
\[\mathcal P  \ll_p  p^{\mathfrak c_1/2 +\mathfrak c_2/2 +i/2 +j/2-\mathfrak c_1/2 -\mathfrak c_2/2} \sum_{x_0\in \mathfrak o^\times / 1+\mathfrak p} \mathcal |P(x_0)|\ll_p  p^{j/2  -\mathfrak c_2/2}.\]  If $j = \mathfrak c_2/2$ and $L$ is split.  By Lemma \ref{split}, we need to consider the case $H_2 = \mathfrak o^\times \backslash (1+\mathfrak p) =\mathfrak o^\times \backslash (1+\mathfrak p \cup 2+\mathfrak p) \bigsqcup_k (2+\varpi ^k \mathfrak o^\times)  $. Then by the Proposition \ref{mainlocal} and Remark \ref{remark} we can deduced to the case $H_2 = \mathfrak o^\times (1+\mathfrak p \cup 2+\mathfrak p)$ and  then obtain  that $\mathcal P \ll  p^{j/2-\mathfrak c_2/2}$. Similarly if $j = \mathfrak c_2/2$ and $L$ is field, then by Lemma \ref{unramifedfieldextension}, Proposition \ref {mainlocal} and Remark \ref{remark} the statement holds.
\end{proof}

In the rest of this section, we are going to prove Proposition \ref{mainlocal}. By the definition, $\mathcal P(x_0)$ is equal to
\begin{equation}
\begin{split}
&\int_{H_2\times H_1 \times \mathfrak p }\psi \left ( (x_0+x) \left ( \frac {f(y) } {\varpi ^i} + \frac {g(z)} {\varpi ^j} + \frac 1 {\varpi ^l} \right )   \right )  \theta^{-1}  (\mathcal L(y))  \eta ^{-1}  (\mathcal M(z)) dx dy dz.\\
& = \int_{\mathcal C }\psi \left ( x_0 \left ( \frac {f(y) } {\varpi ^i} + \frac {g(z)} {\varpi ^j} + \frac 1 {\varpi ^l} \right )   \right )\tilde \theta^{-1}  (\mathcal L(y))  \tilde \eta ^{-1}  (\mathcal M (z)) dy dz,
\end{split}
\end{equation}
where $\mathcal C$ is  defined by 
\[
\left \{ (y,z) \in H_1\times H_2; \frac {f(y) } {\varpi ^i} + \frac {g(z)} {\varpi ^j} + \frac 1 {\varpi ^l} \in \mathfrak p^{-1} \right \}. 
\]
Next we decompose $\mathcal C$ into some pieces $\mathcal C(y,z)$ which is linearly isomorphic to $\mathfrak p^{\lceil i/2\rceil+1} \times \mathfrak p^{j-\nu (g'(z)) -1}$. 

\begin{definition}\label{C0}
Throughout this section, we take $a= \lceil i/2 \rceil+1$ and $b = \lceil j/2 \rceil +1$.
Let $\mathcal C_0\subset \mathcal C$ be a set of representative  of $(\mathcal C+\mathfrak p^a \times \mathfrak p^b)/\mathfrak p^a\times \mathfrak p^b$. For $(y,z) \in \mathcal C_0$, let $\mathcal C(y,z) = \{ (\delta _1,\delta _2) \in \mathfrak p^a\times \mathfrak p^b; (y+\delta _1,z+\delta _2) \in \mathcal C\}$. Then we have decomposition 
\[
    \mathcal C= \bigsqcup _{(y,z)\in \mathcal C_0} (y,z) + \mathcal C(y,z).
\]
\end{definition}

Note that $\theta ^{-1}(L(y))$ is $a$-additive in $H_1$ and $\eta ^{-1} (M(z))$ is $b$-additive in $H_2$. More precisely, we set $\alpha=\alpha _\theta = \alpha _{\tilde \theta }$ be the purely imaginary element associated to $\theta$ by Lemma \ref{character} if $E$ is a field. If $E$ is split and $\theta = \mu \otimes \mu ^{-1} $, we set  $\alpha = (\alpha_\mu  ,0)$ where $\alpha_\mu \in F$ and $\nu (\alpha _\mu ) = - \lceil \mathfrak c_1 /2 \rceil$. Similarly we define $\beta$ as $\alpha _\eta$ or $\alpha _\eta \times (1,0)$ up to $L$.  Then we have  
\[
    \tilde \theta^{-1} (\mathcal L(y+\delta )) = \tilde \theta ^{-1} (\mathcal L(y)) \psi \left (\delta T \left ( \frac {\alpha } {\mathcal L(y)} \right ) \right ) .
\]
We set $\mathcal J (y)= T\left ( \frac {\alpha } {\mathcal L(y)}\right ) $ and $\mathcal K (z) =T\left ( \frac {\beta } {\mathcal M(y)}\right )$. 
\begin{equation*}
\mathcal J (y) = T\left (\frac {\alpha( y+\sqrt D)} {y^2-D}\right )= \frac {2\alpha \sqrt D} {y^2 -D}.
\end{equation*}
Since $\nu (y)\ge 0$, we get that $\nu (\mathcal J(y)) = \nu (\alpha \sqrt D) - \nu (D)=-\lceil \mathfrak c_1/2 \rceil $.
If $E$ is split, recall that we have $\alpha = \alpha _{\mu } (1,0)$ 
\begin{equation*}
\mathcal J (y) = \frac {\alpha_{\mu}} {y-1},
\end{equation*}
Thus $\nu (\mathcal J(y))  = \nu(\alpha _{\mu }) = -\lceil \mathfrak c_1/2 \rceil  $.
Now using the decomposition of $\mathcal C$ we have 
\begin{equation}\label{mathcalPx0}
\begin{split}
    \mathcal P(x_0) &= \sum _{(y,z) \in \mathcal C_0} \psi \left ( x_0 \left ( \frac {f(y) } {\varpi ^i} + \frac {g(z)} {\varpi ^j} + \frac 1 {\varpi ^k} \right )   \right )\tilde \theta ^{-1} (\mathcal L(y)) \tilde \eta ^{-1} (\mathcal M(z))\\
    & \times \int _{\mathcal C(y,z)} \psi \left ( \frac {x_0f'(y)\delta _1 } {\varpi ^i} + \frac {x_0g'(z)\delta _2} {\varpi ^j}+\mathcal J(y)\delta _1+ \mathcal K (z) \delta _2\right ) d\delta _1 d\delta _2
\end{split}
\end{equation}

\begin{proposition}\label{Hyz}For $(y,z)\in \mathcal C_0$,  
We have 
\[
\mathcal C(y,z) = \left \{ \left (\delta _1, \delta -\frac {\delta _1 f'(y) \varpi ^j} {g'(z) \varpi ^i}  \right ); \ \delta _1 \in \mathfrak p^a, \delta \in \mathfrak p^{j-\nu (g'(z)) -1}\right \},
\]
which implies that $\mathcal C(y,z)$ is linearly isomorphic $\mathfrak p^a \times \mathfrak p^{j-\nu (g'(z)) -1}$.
\end{proposition}

\begin{proof}Since $(y,z) \in \mathcal C_0\subset \mathcal C$, we have 
\[
    \frac {f(y) } {\varpi ^i} + \frac {g(z)} {\varpi ^j} + \frac 1 {\varpi ^l} \in \mathfrak p^{-1}.
\]
By the definition and Lemma \ref{Taylor} we have $(\delta _1, \delta _2) \in \mathcal C(y,z)$ if and only if 
\[
    \frac {\delta _1 f'(y)} {\varpi ^i} +\frac {\delta _2g'(z)} {\varpi ^j} \in \mathfrak p^{-1}.
\]
Multiple $\varpi ^j / g'(z)$ we get that 
\[
    \delta _2 +\frac {\delta _1 f'(y) \varpi ^j} {g'(z) \varpi ^i} \in \mathfrak p^{j-\nu (g'(z)) -1}.
\]
Then the statement holds.
\end{proof}

Back to the expression of $\mathcal P(x_0)$. Note that the determinant of the map $(\delta _1,\delta ) \mapsto   \left (\delta _1, \delta -\frac {\delta _1 f'(y) \varpi ^j} {g'(z) \varpi ^i}  \right )$ is $1$, we have 
\begin {equation}\label{integralmathcalC(y,z)}
\begin{split}
&\int _{\mathcal C(y,z)} \psi \left ( \frac {x_0f'(y)\delta _1 } {\varpi ^i} + \frac {x_0g'(z)\delta _2} {\varpi ^j}+\mathcal J(y)\delta _1+ \mathcal K(z) \delta _2\right ) d\delta _1 d\delta _2 \\
& = \int _{ \mathfrak p^a\times \mathfrak p^{j-\nu (g'(z)) -1} } \psi \left ( \left (\mathcal K(z)+\frac {x_0 g'(z)} {\varpi ^j} \right )\delta + \left (\mathcal J(y) -\frac {\mathcal K(z) f'(y)  \varpi ^{j-i} } {g'(z) }\right )\delta _1 \right )  d\delta  d\delta_1. \\
&= p^{-\lceil i/2 \rceil -j+\nu (g'(z)) } \mathbbm{1} _{\mathcal C_1}(z) \mathbbm{1} _{\mathcal C_2}(y,z)
\end{split}
\end{equation}

where 
\begin{equation}\label{C1}
\begin{split}
\mathcal C_1 &= \left \{ z\in H_2; \  \mathcal  K(z)+\frac {x_0 g'(z)} {\varpi ^j}\in \mathfrak p^{1+\nu (g'(z)) -j }\right \} \\
\mathcal C_2 &= \left \{ (y,z) \in H_1\times H_2;  \ \mathcal J(y)-\frac {\mathcal K (z) f'(y)  \varpi ^{j-i} } {g'(z) } \in \mathfrak p^{-a} \right \}.
\end{split}
\end{equation}

\begin{lemma}
(1) $\mathcal C_1$ is nonempty only if $(\romannumeral1)(\romannumeral2)(\romannumeral4)$ of Proposition \ref{mainlocal} hold.\\ 
(2) Assume that $\mathcal C_1$ is nonempty, then $\mathcal C_0$ is nonempty only if $l=\mathfrak c_2-j$.
\end{lemma}

\begin{proof} 
(1) Note that $\mathcal C_1$ is nonempty only if $\nu (\mathcal K(z)) = \nu (x_0g'(z)/\varpi ^j)$. 
For $(\romannumeral 5)$ in proposition \ref{mathcalP},  $L$ is unramified field extension and $j = \mathfrak c_2/2$. Then $\nu (x_0g'(z)/\varpi ^j) = -3\nu (z)-\mathfrak c_2/2 > \nu (\mathcal K (z)) = -2 \nu (z) - c(\eta ) $. Thus $\nu(\mathcal K(z) + x_0g'(z)/\varpi ^j) = \nu (\mathcal K(z ))$. Thus $\mathcal C_1$ is empty.\\
For $(\romannumeral 3)$ in proposition \ref{mathcalP}, we have  $L$ is split and $z\in 2+\varpi ^k \mathfrak o$ for some positive $k$. Then $\nu (g'(z)/\varpi ^j) = k-\mathfrak c_2/2 > \nu (\mathcal K(z))$. Thus $\mathcal C_1 =\emptyset$. \\
(2) Now we only care about cases $(\romannumeral 1)(\romannumeral 2)(\romannumeral 4)$  where we have  $i-\mathfrak c_1 = \nu (f(y)/\varpi ^i)>\nu (g(z)/\varpi ^j) = j- \mathfrak c_2$. Thus $\mathcal C_0$ is nonempty only if $l= \mathfrak c_2-j$.
\end{proof}

For now on we assume that  $l= \mathfrak c_2-j$ and $H_2$ is   one of $(\romannumeral 1)(\romannumeral 2)(\romannumeral 4)$ where we have $\nu (g'(z))= \nu (D_2) +\nu (z)=j -\lceil \mathfrak c_2/2 \rceil$. Then we have 
\begin{proposition} Suppose $\mathcal P(x_0)$ is nonzero. Combine (\ref{mathcalP}) and (\ref{integralmathcalC(y,z)}) we get that 
\[
| \mathcal P(x_0) | \le p^{-\lceil i/2 \rceil -\lceil \mathfrak c_2/2 \rceil  } \# \mathcal C_0 \cap \mathcal C_1 \cap \mathcal C_2.
\]
\end{proposition}
In the rest of this section, we are going to prove the following theorem.
\begin{theorem}\label{bb} We have 
\[
\# \mathcal C_0 \cap  \mathcal C_2 \ll _p 1.
\]
\end{theorem}

\begin{proposition}\label{characteristicC0}
    We have $(y,z) \in \mathcal C+\mathfrak p^a\times \mathfrak p^b$  if and only if 
    \[
        \frac {f(y)} {\varpi ^i} +\frac {g(z)} {\varpi ^j} +\frac 1 {\varpi ^l} \in \mathfrak p^{\lceil j/2\rceil +1 -\lceil \mathfrak c_2 /2 \rceil  }.
    \]
\end{proposition}

\begin{proof}
By the definition $(y,z) \in \mathcal C+\mathfrak p^a\times \mathfrak p^b$ if and only if there exists $(\delta _1, \delta _2)\in \mathfrak p^a\times\mathfrak p^b$ such that $(y+\delta _1, z+\delta _2) \in \mathcal C$, 
By Lemma \ref{Taylor}, it is equivalent to say that
\[
    \frac {f(y)} {\varpi ^i} +\frac {g(z)} {\varpi ^j} +\frac 1 {\varpi ^l} + \frac {\delta _1 f'(y)} {\varpi ^i} +\frac {\delta _2g'(z)} {\varpi ^j} \in \mathfrak p^{-1}.
\]
Note that $\nu (f'(y)) = \nu (y) +\nu (D_1) =i-\lfloor \mathfrak c_1/2 \rfloor  + \nu (D_1) $ and thus $\delta _1 f'(y) / \varpi ^i$ generates  $\mathfrak p^{a-\lceil \mathfrak c_1/2  \rceil }$ as $\delta _1 \in \mathfrak p^a$. And similarly $\delta _2g'(z)/\varpi ^j$ generates $\mathfrak p^{b-\lceil \mathfrak c_2 /2 \rceil  }$. Thus the statement holds.
\end{proof}

We now present an identity that  for variable substitution.
\begin{lemma}\label{change}
Let  $\alpha _1 = \alpha\in E$ be purely imaginary if $E$ is field and $\alpha =  (\alpha _1,0)$ if $E$ is split. Then we have 
\begin{equation*}
\left (\frac {f'(y)} {\mathcal J (y)} \right )^2 = \frac {f(y)(f(y)-1)} {\alpha_1 ^2}.
\end{equation*}
\end{lemma}
\begin{proof}
By direct computation.
\end{proof}

Recall Proposition \ref{characteristicC0} and the definition of $\mathcal C_2$ (\ref{C1}), $\# \mathcal C_0\cap \mathcal C_2$ is the number of $(y,z)\in H_1\times H_2  \mod \mathfrak p^a \times \mathfrak p^b$ which satisfies 
\begin{equation}\label{aa}
\begin{dcases}
&\frac {f(y)} {\varpi^i } + \frac {g(z)} { \varpi ^j} +\frac 1 {\varpi ^l} \in \mathfrak p^{\lceil j/2 \rceil +1 - \lceil \mathfrak c_2 /2 \rceil }\\
&\frac {f'(y)}  {\varpi^i \mathcal J(y)}   -\frac {g'(z)} {\varpi ^j \mathcal K(z)}   \in \mathfrak p^{\lceil \mathfrak c_1/2 \rceil-\lceil i/2 \rceil -1 }.
\end{dcases}
\end{equation}

\begin{lemma}
We claim that 
\[
    \frac {f'(y)}  {\varpi^i \mathcal J(y)}   +\frac {g'(z)} {\varpi ^j \mathcal K(z)}\in \mathfrak o^\times .
\]
\end{lemma}

\begin{proof}
By simple computation we know that $\nu (f'(y) / \varpi^i \mathcal J(y)) = \nu (g'(z)/\varpi ^j \mathcal K(z))=0$. But by the second line of  (\ref{aa}) there exists valuation dropping. Thus $\nu (f'(y) / \varpi^i \mathcal J(y) +g'(z)/\varpi ^j \mathcal K(z) )=0$.
\end{proof}

Now we change the variables by $Y= f(y) / \varpi ^i$ and $Z= g(z) / \varpi ^j$.
Multiply the second line of  (\ref{aa}) by 
\[
    \frac {f'(y)}  {\varpi^i \mathcal J(y)}   +\frac {g'(z)} {\varpi ^j \mathcal K(z)},
\]
and apply Lemma \ref{change} we get that
\begin{equation*}
\begin{dcases}
    & Y+Z+\varpi ^{-l} \in \mathfrak p^{\lceil j/2 \rceil +1 - \lceil \mathfrak c_2 /2 \rceil},\\
    & \frac {Y(Y-\varpi ^{-i} )}{\alpha _1 ^2} -\frac {Z(Z-\varpi ^{-j} )} {\beta _1 ^2} \in \mathfrak p^{\lceil \mathfrak c_1/2 \rceil-\lceil i/2 \rceil -1 }.
\end{dcases}
\end{equation*}
We verify that this changing is injective.
\begin{proposition}
For any $Y\in \varpi ^{i-\mathfrak c_1} \mathfrak o^\times $, there is at most one $y \in  H_1/\mathfrak p^{\lceil i/2 \rceil +1} $ such that
\begin{equation*}
Y-f(y)/\varpi ^i  \equiv 0 \mod  {\mathfrak p^{\lceil i/2 \rceil +1-\lceil \mathfrak c_1/2 \rceil }}.
\end{equation*}
\end{proposition}

\begin{proof}
It is sufficient to prove that if $f(y_1)/ \varpi ^i -f(y_2) / \varpi ^j\in \mathfrak p^{\lceil i/2 \rceil+1 -\lceil \mathfrak c_1/2 \rceil }$, then $y_1- y_2 \in \mathfrak p^{\lceil i/2 \rceil +1}$. By the Lemma \ref{Taylor}, 
\[
f(y_1)/ \varpi ^i -f(y_2) / \varpi ^i = (y_1-y_2) f'(y_2) /\varpi ^i + (y_1-y_2)^2 h(y_2,y_1-y_2) \in \mathfrak p^{\lceil i/2 \rceil + 1-\lceil \mathfrak c_1/2 \rceil }.
\]
It is easy to check that the last term always  belongs to $\mathfrak p^{\lceil i/2 \rceil + 1-\lceil \mathfrak c_1/2 \rceil }$. Thus $y_1-y_2 \in \varpi ^i f'(y)^{-1} \mathfrak p^{\lceil i/2 \rceil +1-\lceil \mathfrak c_1/2 \rceil } = \mathfrak p^{\lceil i/2 \rceil +1}$.
\end{proof}

Then Theorem \ref{bb} is reduced to the following lemma. 
\begin{lemma}
The number of  the solutions of relations (\ref{aa}) for $Y \mod  \mathfrak p^{\lceil i/2 \rceil + 1-\lceil \mathfrak c_1/2 \rceil }$ and $Z \mod   \mathfrak p^{\lceil j/2 \rceil + 1-\lceil \mathfrak c_2/2 \rceil }$ is $O(1)$.
\end{lemma}

\begin{proof}
Suppose $(Y,Z)$ satisfy (\ref{aa}).
The main idea is to replace $Z$ in the second line by $-Y-\varpi ^{-l}$. More precisely, we claim that 
\[
    \frac {Z(Z-\varpi ^{-j} )} {\beta _1 ^2} -\frac {(Y+\varpi ^{-l})(Y+\varpi ^{-j}+\varpi ^{-l} )} {\beta _1 ^2}\in \mathfrak p^{\lceil \mathfrak c_1/2 \rceil-\lceil i/2 \rceil -1 }.
\]
From the first line of (\ref{aa}) we can write $Z = -Y-\varpi ^{-l} + O(\mathfrak p^{\lceil j/2 \rceil + 1-\lceil \mathfrak c_2/2 \rceil })$. We temporarily write $-Y-\varpi ^{-k}$ as $Z_1$. Then 
\[
    Z(Z-\varpi ^{-j}) = Z_1(Z_1-\varpi ^{-j}) + (2Z_1-\varpi ^{-j}) O(\mathfrak p^{\lceil j/2 \rceil + 1-\lceil \mathfrak c_2/2 \rceil }) + O(\mathfrak p^{2\lceil j/2 \rceil + 2-\lceil 2\mathfrak c_2/2 \rceil }).
\]
Now we need to show the last two terms belong to $\beta ^2 \mathfrak p^{\lceil \mathfrak c_1/2 \rceil - \lceil i /2\rceil -1} = \mathfrak p^{\lceil \mathfrak c_1/2 \rceil - \lceil i /2\rceil-\mathfrak c_2 -1}$.
Note that $\nu (2Z_1+\varpi ^{-j}) = -j, \lceil j/2  \rceil +1- \lceil \mathfrak c_2/2 \rceil \ge j/2 -\mathfrak c_2/2$ and $\lceil \mathfrak c_1/2 \rceil - \lceil i/2 \rceil -1 \le \mathfrak c_1 /2 -i/2 \le \mathfrak c_2/2 -j/2$. We only need to show 
$-j +j/2-\mathfrak c_2/2 \ge \mathfrak c_2/2 -j/2 -\mathfrak c_2$ and $j-\mathfrak c_2\ge \mathfrak c_2/2 -j/2-\mathfrak c_2$ which are clear since $j\ge \mathfrak c_2/2$. Now it is sufficient  to consider the congruence equation
\begin{equation*}
\frac {Y(Y-\varpi ^{-i} )}{\alpha _1 ^2}-\frac {(Y+\varpi ^{-l})(Y+\varpi ^{-j}+\varpi ^{-l} )} {\beta _1 ^2}= a_1Y^2 +a_2 Y+a_3\in \mathfrak p^{\lceil \mathfrak c_1/2 \rceil-\lceil i/2 \rceil -1 },
\end{equation*}
where 
\begin{equation*}
\begin{split}
& a_1 = \frac 1 {\alpha_1^2} - \frac 1 {\beta _1^2} \in \varpi  ^{\max \{ \mathfrak c_1, \mathfrak c_2\} } \mathfrak o^\times = \varpi ^ {\mathfrak c_2} \mathfrak o^\times , \\
& a_2 = -\frac 1 {\alpha ^2 \varpi ^i} - \frac 2 {\beta ^2 \varpi ^l } - \frac 1 {\beta^2 \varpi ^j} \in \varpi ^{\mathfrak c_1-i} \mathfrak o^\times.
\end{split}
\end{equation*}
The last line is because that $\nu (1/\alpha ^2\varpi ^i) = \mathfrak c_1-i < \nu (1/\beta ^2 \varpi ^j )= \mathfrak c_2-j < \nu (2/\beta ^2\varpi ^l) = j$.
Suppose $Y_1$, $Y_2$ are different solutions, then we can get
\[
(Y_1- Y_2) (a_2 + a_1(Y_1+Y_2)) \in \mathfrak p^{\lceil \mathfrak c_1/2 \rceil-\lceil i/2 \rceil -1 }.
\]
Since $\nu (a_1(Y_1 + Y_2)) > i \ge \mathfrak c_1-i =\nu (a_2)$, we can get $\nu (a_2 +a_1 (Y_1+Y_2)) = \mathfrak c_1-i$ and $Y_1 - Y_2 \in \mathfrak p ^{\lfloor i/2 \rfloor - \lfloor \mathfrak c_1/2 \rfloor -1}$. Thus there is at most $p^{2+2\{ i/2 \}-2\{ \mathfrak c_1/2 \} }$'s solutions for $Y \mod \mathfrak p ^{\lceil i/2 \rceil - \lceil \mathfrak c_1/2 \rceil +1 }$ and the statement holds.   
\end{proof}

\section{Voronoi summation formula}

In this section we establish an explicit Voronoi summation formula. Let $\pi$ be an irreducible cuspidal automorphic representation of $\mathrm {GL}_2$ over $\mathbb Q$ with trivial central character.  For any $\phi \in \pi$ the associated Whittaker function is given by the global Fourier coefficient 
\[
W_\phi (g) = \int _{N(\mathbb Q)\backslash N(\mathbb A)} \phi(ng) \psi^{-1} (n) \mathrm d n.
\]

\begin{lemma}[\cite{Ass21}, Lemma 2.1]
Let $\zeta\in \mathbb A$ and let $\phi\in \pi$. Then we have the adelic Voronoi summation formula 
\begin{equation}\label{adelicvoronoi}
\sum_{\gamma\in \mathbb Q^*} \psi(\gamma \zeta)W_\phi  (a(\gamma )) = \sum_{\gamma\in \mathbb Q^*} \widetilde {W}_\phi (a(\gamma ){^Tn}(-\zeta ))
\end{equation}
where $\widetilde {W}_\phi(x)= W_\phi  (\omega {^T x^{-1}})$ and $^T x$ is the transpose of $x$.
\end{lemma}

\subsection{Setting up the left side}
We choose $\phi$ such that 
\begin{equation*}
W_\phi = \prod _{\nu} W_{\nu}
\end{equation*}
is a pure tensor. Thus we can treat each place on its own.
Since the Kirillov model of $\pi_\infty$ contains the space of Schwartz functions we can choose $W_{\infty}$ such that, for all $\gamma \in \mathbb R^ \times$, we have 
\begin{equation*}
W_{\infty}(a(\gamma)) = \vert \gamma \vert ^{\frac 1 2}  h (\gamma) 
\end{equation*}
for $h \in C_c^\infty (\mathbb R^*)$ with support in $\mathbb R_+$. 
At all the finite places we choose $\phi$ such that $W_l$ is normalized Whittaker function for new form. Then by \cite[Lemma 2.5]{Saha16}, $W_l(\gamma)$ is supported on $\mathfrak o_l$ and for $\gamma \in \mathfrak o_l$ we have   
\begin{equation*}
W_l(a(\gamma))= \lambda _\pi (p^{\nu _l(\gamma)})|\gamma|_l^{\frac 1 2} .
\end{equation*}
If $\pi_l$ is ramified with conductor $\mathfrak c\ge 2$, then $W_l(a(x))$ is $\mathfrak o_l^\times$-invariant and it is supported on $\mathfrak o_l^\times$ by Lemma \ref{Whittaker function}. Thus $W_l(a(x))$ is the characteristic function on $\mathfrak o_l^\times$. 

Now we fix a place $p$ and choose $\zeta_p = \frac u {p^k}$. 
For finite place $l \ne p$ we set $\zeta _l =\frac b d$,
where $b$, $d$ are integers, $(d,p) =1$  and set $\zeta _\infty  =0$. By product formula we have for $n$ integer, 
\begin{equation*}
\psi(n\zeta)=\psi _p\left (\frac {nu} {p^k}\right ) e \left ( \frac {n b} {d} \right).
\end{equation*}
We also assume that for $l\ne p$, $\pi_l$ is unramified and $\pi_p$ is ramified. Then $W_\pi (a(\gamma))$, $\gamma\in \mathbb Q$ is nonvanishing only if $\gamma =n \in \mathbb N^*$. Then under those settings the left side of (\ref{adelicvoronoi}) is given by 
\begin{equation}\label{left}
    \sum_{\gamma\in \mathbb Q^\times} \psi(\gamma \zeta ) W_\pi (a(\gamma) ) = \sum _{n\in \mathbb N^*} e \left (\frac {nb} {d}\right ) \psi _p \left (\frac {n u} {p^k} \right ) \lambda _ \pi (n) h(n).
\end{equation}

\subsection{Computation the right hand side}
From \cite[(2.3)]{Ass21}, the right hand side is 
$$
\sum_{\gamma\in \mathbb Q}\prod _{\nu\le \infty }W_{\nu}(a(\gamma)\omega n(\zeta)).
$$

If $\pi_l$ is unramified and $\zeta_l\in \mathbb Z_l$, then 
$$
W_p (a(\gamma)\omega n(\zeta))= W_p (a(\gamma))=
\begin{cases}
    \vert \gamma \vert ^ {\frac 1 2} \lambda _{\pi} ( p ^{\nu_p(\gamma)})  & \mathrm {if} \ \gamma \in \mathbb Z_p\\
    0  & \mathrm {else}.
\end{cases}
$$

If $\zeta_l\notin \mathbb Z_l$, then by \cite[Section 2.2.1]{Ass21},
$$
W_l (a(\gamma)\omega n(\zeta))=
\begin{cases}
    \psi _l\left ( -\frac \gamma {\zeta_l} \right ) \left \vert \frac \gamma {\zeta_l^2} \right \vert ^ {\frac 1 2} \lambda _{\pi} \left ( l ^{\nu_l(\gamma/ \zeta_l^2)} \right )  & \mathrm {if} \ \gamma/\zeta_l^2 \in \mathbb Z_l\\
    0  & \mathrm {else}. 
\end{cases}
$$
For the archimedean place, by \cite[Sect 2.2.4]{Ass21}, and assume that  $\pi_\infty$ is  a discrete series representation   of weight $\kappa \ge 2$. Then we have
\[
W_{\infty} (a(\gamma)\omega n(\zeta_\infty)) = 
\begin{dcases}
\vert \gamma \vert  ^{\frac 1 2} \mathcal H h ( \gamma  ), \ \ &\gamma >0\\
0, \ \ &\gamma <0.
\end{dcases}
\]
where 
\begin{equation*}
\mathcal H h (y) := 
 2\pi i^\kappa  \int _{\mathbb R^+}   J_{\kappa -1} (4\pi\sqrt {xy} ) h( x).
\end{equation*}

We set $S_1$ be the set of finite place $l$ such that $l \mid d$. And set $S_2$ be the set of finite places outside $S_1\cup \{ p\}$. Then  we define 
\begin{equation*}
\prod _{\nu }W_{\nu } = W_{\infty} W_p \prod _{l \in S_1}W_l \prod _{l \in S_2}W_l:= W_\infty W_p W_1 W_2.
\end{equation*}
Suppose $\gamma = p^t \gamma_1 \gamma_2$ where $\gamma_1 = \prod _{l\in S_1} l^{\nu_l(\gamma)}$ and $\gamma_2 = \prod _{l\in S_2} l^{\nu_l(\gamma)}$. Then $W_1$
is nonzero only if $\gamma_1 d^2 \in \mathbb Z$ and in that case we have 
\begin{equation*}
W_1=\frac {\lambda _\pi (\gamma_1d^2)} {\sqrt {\gamma_1d^2}} \prod_{l\in S_1} \psi_l\left (\frac {\gamma d}  b\right ) .
\end{equation*}
And similarly we can see $W_2$
is nonzero  only if $\gamma_2 \in \mathbb Z$ in that case we have 
\begin{equation*}
W_2= \frac {\lambda _\pi (\gamma_2)} {\sqrt {\gamma_2}} .
\end{equation*}
In conclusion, $W_1W_2$ is nonzero only if $ \gamma_1\gamma_2d^2 \in \mathbb Z$. Let $\bar b$ be the inverse of $b \mod d$, then for every $l\in S_1$, 
\begin{equation*}
\frac {\gamma d} b -  \gamma d \bar b  =\frac {1-\bar b b} {d} \frac {\gamma d^2} {b} \in \mathbb Z_l.
\end{equation*}
Now we set $\gamma_1\gamma_2 d^2 =n$. Then by the product formula we have 
\begin{equation*}
\prod_{l\in S_1} \psi_l\left (-\frac {\gamma d}   b\right )=\prod_{l\in S_1} \psi_l\left (-\gamma d \bar b\right ) = 
\begin{dcases}
e\left (-\frac {p^tn\bar b} d \right ) & \ \ \ \ \mathrm{if} \ t\ge 0 \\
e\left(-\frac {n {\bar p}^{-t}\bar b} d\right) & \ \ \ \ \mathrm{if} \ t<0,
\end{dcases}
\end{equation*}
Where $\bar p$ and $\bar b$ means the inverse of $p$ and $b \mod d$. We can give the right side:
\begin{equation*}
\begin{split}
&\sum_{\gamma\in \mathbb Q^\times }\prod _{\nu\le \infty } W_{\nu}(a(\gamma)\omega n(\zeta))=\\
&\frac 1 d \sum _{ n\in \mathbb N^*, (n,p)=1} \lambda _\pi (n)\sum _{t\in \mathbb N} p^{\frac t 2} e\left (\frac {p^t n\bar b}  d\right ) W_p\left (a\left (\frac {p^t n} {d^2}\right ) \omega n\left ( \frac u {p^k}\right )\right ) \mathcal H h \left (\left \vert \frac {p^t n} {d^2} \right \vert \right )\\
&+ \frac 1 d \sum _{ n\in \mathbb N^*, (n,p)=1} \lambda _\pi (n)\sum _{t\in -\mathbb N^*} p^{\frac t 2} e\left (\frac {\bar p^{-t} n\bar b}  d\right ) W_p\left (a\left (\frac {p^t n} {d^2}\right ) \omega n\left ( \frac u {p^k}\right )\right ) \mathcal H h \left (\left \vert \frac {p^t n} {d^2} \right \vert \right ).
\end{split}
\end{equation*}
We use the following lemma to translate the Whittaker function to the form we want.
\begin{lemma}\label{whittaker} 
For $k\ge 0$, $u \in \mathfrak o^\times$ we have
\begin{equation*}
\begin{split}
W_p \left (a ( x) \omega n( u /\varpi ^k)\right ) =
\psi_p \left (-\frac {x\varpi ^k} u\right ) W^{(k)}_p \left (\frac {x\varpi ^{2k}} u\right ).
\end{split}
\end{equation*}
\end{lemma}

\begin{proof}
Note that 
\begin{equation*}
\begin{split}
   &a(x) \omega n(u/\varpi ^k) 
   \left (\begin{array}{cc}
     u  &  0 \\
     0  & 1
  \end{array} \right )=\\
   & \left (\begin{array}{cc}
     -u/\varpi ^k  &  0 \\
     0  & -u/\varpi ^k
  \end{array} \right )
  \left (\begin{array}{cc}
     1  &   -x \varpi ^k/u \\
     0  & 1
  \end{array} \right )\left (\begin{array}{cc}
     x\varpi ^{2k}/u  &  0 \\
     0  & 1
  \end{array} \right )
  \left (\begin{array}{cc}
     1  &  0 \\
     \varpi ^k  & 1
  \end{array} \right ),
  \end{split}
\end{equation*}
and $\left (\begin{array}{cc}
     u  &  0 \\
     0  & 1
  \end{array} \right )\in K_1(\mathfrak c)$ and $W_p$ is center-invariant. Thus the statement holds.
\end{proof}
By Lemma \ref{whittaker}, 
\begin{equation*}
W_p\left (a\left (\frac {\varpi ^t n} {d^2}\right ) \omega n\left ( \frac u { \varpi ^k}\right )\right )=W_p^{(k)} \left (\frac {\varpi ^{2k+t} n} {d^2 u} \right ) \psi _p \left ( - \frac {\varpi ^{k+t} n} {d^2 u} \right ).
\end{equation*}
Then by Lemma \ref{Whittaker function}, we get that for $\mathfrak c/2 \le k< \mathfrak c$, $W_p$ is nonzero only if  $t \ge  -2k$.  Then we replace $p^{2k+t} n$ by $n$. For $0 < k< \mathfrak c/2 $, by Lemma \ref{Atkin}, we have
\begin{equation*}
W_p\left (a\left (\frac {\varpi ^t n} {d^2}\right ) \omega n\left ( \frac u {\varpi ^k}\right )\right )=a_kW_p^{(\mathfrak c-k)} \left (-\frac {\varpi ^{t+\mathfrak c} n} {d^2 u} \right ) ,
\end{equation*}
which is nonzero only if $t= - \mathfrak c$ and we replace $p^{\mathfrak c+t} n$ by $n$.

Finally we get following special type Voronoi  summation formula: 
\begin{theorem} \label{Voronoi}
Let $\pi$ be a cuspidal automorphic representation with conductor $p^\mathfrak c$ and trivial central character. Let $h \in {C_c}^\infty (\mathbb R^\times _{+})$ and $W_p$ be p-adic normalized whittaker new function of $\pi_p$. Let $b$,$d$,$k \in \mathbb N_+$ such that  $(b,d) = (d,p) =1$ and $u\in \mathfrak o_p^\times $ and set $c = dp^k$. Then for $\mathfrak c/2\le k\le \mathfrak c-1$  
\begin{equation}\label{kbig}
\begin{split}
&\sum _{n\in \mathbb N^*} e \left (\frac {nb} {d}\right ) \psi_p \left (\frac {n u} {p^k} \right ) \lambda _ \pi(n) h (n) =\\
&  \sum _{ n\in \mathbb N^*} \frac {1} {c|n|_p^{1/2}} \lambda _\pi (n)e\left (-\frac {{ \overline {p^{2 k} b}n}}  { d}\right )
W_p^{(k)} \left ( \frac  n {d^2 u} \right ) \psi _p \left (-\frac n {d^2up^k} \right ) \mathcal H h \left (\left \vert \frac { n} {c^2} \right \vert \right ).
\end{split}
\end{equation}
For $0<k< \mathfrak c/2$  we have
\begin{equation*}
\begin{split}
&\sum _{n\in \mathbb N^*} e \left (\frac {nb} {d}\right ) \psi_p \left (\frac {n u} {p^k} \right ) \lambda _ \pi(n) h (n) =\\
&\frac \eta {d p^{\mathfrak c/2}} \sum _{ n\in \mathbb N^*, (n,p)=1} \lambda _\pi (n)
e\left (-\frac {  \overline {p^{2\mathfrak c}b} n}  { d}\right ) W_p^{(\mathfrak c-k)} \left (-\frac {n} {d^2u} \right ) \mathcal H h \left (\left \vert \frac {n} {d^2p^{\mathfrak c}} \right \vert \right ).
\end{split}
\end{equation*}
\end{theorem}


Last we introduce a reslut about the bound of Hankel transform.
\begin{lemma}[\cite{Holowinsky2014}, Lemma 2]\label{arch}
    Let $\kappa , \iota \ge 2$ be fixed integers and let $a,b >0$. Define 
    $$
I(a,b):= \int _0 ^\infty \frac {h(\xi)}{\sqrt {\xi}} J_{\kappa -1} (4\pi a \sqrt {\xi}) J_{\iota -1} (4\pi b\sqrt {\xi}) d\xi 
    $$
    where $h$ is a smooth function supported on $[\frac 1 2 , \frac 5 2]$ with bounded derivatives. We have 
    $$
    I(a,b) \ll _{j,h} \frac a {(1+a)^{3/2}} \frac b {(1+b)^{3/2}} \vert a-b \vert ^{-j}
    $$
    for any $j \ge 0$.
\end{lemma}

\section{Complete the proof} \label{section5}
Now Let $\mathcal F_\theta [l]$ be the set of holomorphic newforms of level $p^{\mathfrak c_1}$ whose associated local representation $\pi_p \in \pi_\theta [l]$, where $\pi_\theta$ has trivial central character.
Recall Section \ref{subsetion}, let $g$ be  a cusp newform of level $p^{\mathfrak c_2}$ with trivial nebentypus. Assume $\mathfrak c_1<\mathfrak c_2$. We set $N = p^{\mathfrak c_2} = \sqrt {C(f\times g)}$.
Consider the first moment of the Rankin-Selberg $L$-functions
\begin{equation}\label{M}
M=  \frac {\Gamma (\kappa -1)} {(4\pi)^{\kappa - 1} }\sum _{f\in \mathcal F_\theta [l]} \frac {L(f\times g, 1 /2)}{\Vert f \Vert ^2 } = \sum_{f\in \mathcal F_\theta [l]} ^* L(f\times g, 1/2).
\end{equation}
Here $\sum ^*_f$ is the harmonic average:
\[
    \sum ^*_f \alpha _f :=  \frac {\Gamma (\kappa -1 )}{(4\pi)^{\kappa -1} } \sum _f \frac {\alpha _f} {\| f\|^2}
\]
 $f$ is normalized so that $\lambda _f (1)=1$.
Applying approximate functional equation (\ref{Approxomate}) and  Proposition \ref{epsilon_factor_independence}, we get that
\begin{equation*}
\begin{split}
M&= \sum_{f\in \mathcal F_\theta [l]} ^* \left ( \sum _{n\ge 1} \frac {\lambda _f(n) \lambda _g(n)} {\sqrt n} W_{1/2} \left (\frac {n} {N } \right ) +  \sum _{n\ge 1} \frac {\bar \lambda _f(n) \bar \lambda _g(n)} {\sqrt n} \widetilde {W}_{1/2}\left (\frac {n} {N} \right ) \right ) \\
& =: \sum_{f\in \mathcal F_\theta [l]}  ^*   Z_1+Z_2 
\end{split}
\end{equation*}

We set $M_0 = \sum_{f\in \mathcal F_\theta [l]} ^*   Z_1$. Our aim is to estimate $M_0$, and the sum of $Z_2$ is just similar.
Multiplying with $\bar \lambda _f (1) =1$ and applying the refined Petersson trace formula, Theorem  \ref{trace formula}. We get that 
\begin{equation*}\label{M0}
M_0 = M^d + M^{od}
\end{equation*}
where $M^d$ involves diagonal terms coming from $\delta _{m_1 =m_2}$:
$$
M^d = W_{1/2}\left ( \frac 1 N \right ) C_\mathcal F [l],
$$
and $M^{od}$ involves the off-diagonal terms:
\begin{equation}\label{Mod}
\begin{split}
M^{od} &= 2\pi i^\kappa  C_{\mathcal F} [l]  \sum _{c_l \mid  c} \frac 1 c \sum _n  \frac { \lambda_g (n)  }{\sqrt n} W_{1/2} \left ( \frac n N \right )G \left (n,1,\theta , \frac 1 {c^2} \right )J_{\kappa -1}\left (\frac {4\pi \sqrt n} c \right )\\
&=: 2\pi i^\kappa  C_{\mathcal F} [l]  \sum _{c_l \mid  c} \frac  {A_c} c\asymp p^{\mathfrak c_1 +l } \sum _{c_l | c} \frac {A_c} c.
\end{split}
\end{equation}

We break the sum in $n$ into dyadic ranges as usual by multiplying with a bump function $\eta _Z$:
\begin{equation*}
A_{c,Z} =  \sum_n\frac { \lambda_g (n)  }{\sqrt n} W \left ( \frac n N \right )J_{\kappa -1}\left (\frac {4\pi \sqrt n} c \right ) G\left (n,1,\theta,\frac 1 {c^2} \right )\eta _Z(n), 
\end{equation*}
where the size of the sum in $n$ is $Z \ll N^{1+\varepsilon}$. Up to a small error, we may also assume that $c\ll N^A$ for some fixed large $A$. This is for the complementary range,one can
 easily control the sum by using the square root cancellation of $G\left (n,1,\theta ,\frac 1 {c^2}\right )$ (see  \cite[Lemma 4.7, Lemma 4.14]{Hu23}) and the fact that when $\kappa \ge 2$, 
 \[
     J_{\kappa -1} (x) \ll x \ \mathrm {as} \ x\to 0.
 \]
Then  $\sum \frac 1 c$ is controlled by $\log N$. 
\begin{proposition} \label{Acz}
We have 
\begin{equation*}
A_{c,Z} \ll_\epsilon  N^{\epsilon} p^{\mathfrak c_2/2-k/2}.
\end{equation*}
\end{proposition}
The proof mainly follows \cite[Lemma 6.3]{Hu23}. Form now on, for conciseness we drop
 all $\epsilon $-terms in our computations.
\begin{proof}
Let $W_1$,$W_2$ be $p$-adic normalized Whittaker functions of $f$, $g$.  Set $c=p^k d$ where $(p,d)=1$. Then by the Lemma \ref{generalkloosterman} we  have 
\begin{equation*}
\begin{split}
G\left (  n,1,\theta, \frac 1 {c^2} \right ) &= O(p^k) \sum _{y\in (\mathbb Z/ d\mathbb Z)^\times } e\left ( \frac {\bar p^{2k} y+n \bar y} {d}  \right )\\
&\times  \int _{\mathfrak o^\times } W_1 \left ( a\left (\frac 1 {c^2} \right ) \omega n\left (\frac u {p ^k} \right ) \right ) \psi_p \left ( -\frac {nu} {p ^k}\right ) du.
\end{split}
\end{equation*}
Then we have 
\begin{equation*}
       A_{c,Z} = O(p^k) \sum _{y\in (\mathbb Z/ d\mathbb Z)^\times } e\left ( \frac {\bar p^{2k} y} d  \right ) \int _{\mathfrak o_p^\times  } W_1 \left (a\left (\frac 1 {c^2} \right ) \omega n\left (\frac u {\varpi^k} \right ) \right ) K_{c,Z}(y,u)du
\end{equation*}
where
\begin{equation*}
K_{c,Z} (y,u) = \sum_n \frac { \lambda_g (n)  }{\sqrt n} W_{1/2}\left (\frac n N\right )J_{\kappa -1}\left (\frac {4\pi \sqrt n} c \right )e \left (\frac {n\bar y} d \right ) \psi_p\left (-\frac {nu} {\varpi^k}\right ) \eta_Z(n).
\end{equation*}
Now we apply the Voronoi summation formula Theorem \ref{Voronoi} to the sum $K_{c,Z} (y,u)$ by setting
\[
    h(n) = W_{1/2}\left (\frac n N\right )J_{\kappa -1}\left (\frac {4\pi \sqrt n} c \right ) \eta_Z(n).
\]
\textbf{Case 1}. $\mathfrak c_1/2+l \le k \le \mathfrak c_2/2$. Then
\begin{equation} \label{a1}
    \begin{split}
       A_{c,Z} &= O(p^{k-\mathfrak c_2/2}/d) \sum _{ n\in \mathbb Z, (n,p)=1} \lambda _\pi (n)\sum _{y\in (\mathbb Z/ d\mathbb Z)^\times } e\left ( \frac { (\bar p^{2k}-\bar p^{2\mathfrak c_2}n)y} {d}  \right )\\
       &\int _{\mathfrak o_p^\times }
  W_1 \left ( a\left (\frac 1 {c^2} \right ) \omega n\left (\frac u {\varpi^k} \right ) \right ) 
 W_2 ^{(\mathfrak c_2-k)} \left (\frac n {d^2 u} \right ) du 
 \\
 &\mathcal H h \left (\left \vert \frac { n} {p^{\mathfrak c_2} d^2} \right \vert \right ) 
\end{split}
\end{equation}
By Lemma \ref{whittaker} we have 
\begin{equation*}
W_1 \left ( a\left (\frac 1 {c^2} \right ) \omega n\left (\frac u {\varpi^k} \right ) \right )= \psi_p\left ( - \frac {\varpi^k} {c^2 u } \right ) W_1^{(k)} \left ( \frac 1 {d^2 u} \right ).
\end{equation*}
Then by Theorem \ref{mainlocal}, the integral in the (\ref{a1}) is bounded by $p^{(\mathfrak c_2-k)/2-\mathfrak c_2/2+O(1)} = p^{-k/2+O(1)}$.
For archimedean part by the definition we have  
\begin{equation*}
\begin{split}
 \mathcal H h  (\left \vert n/d^2p^{\mathfrak c_2} \right \vert  )&\ll \int _0 ^\infty \frac {\eta \left (  \frac x Z \right )} {\sqrt x} J_{\kappa -1} \left ( \frac {4\pi\sqrt x } c\right ) J_{\iota -1}\left ( \frac {4\pi\sqrt {nx} } {dp^{\mathfrak c_2/2}}\right ) dx \\
&= \sqrt Z\int _0 ^\infty  \frac {\eta  (   x   )} {\sqrt x} J_{\kappa -1} \left ( \frac {4\pi\sqrt {Zx} } c\right ) J_{\iota -1}\left ( \frac {4\pi\sqrt {nZx} } {dp^{\mathfrak c_2/2}} \right ) dx \\
&= \sqrt Z I\left (\frac {\sqrt Z} c, \frac {\sqrt {Zn}} {dp^{\mathfrak c_2/2}} \right ).
\end{split}
\end{equation*}
Then by the Lemma \ref{arch}, the function $\mathcal H h  (\left \vert n/d^2p^{\mathfrak c_2} \right \vert  )$ restricts the sum to essentially  
\begin{equation*}
\left \vert 1 -\sqrt {\frac {n} {p^{\mathfrak c_2-2k}}} \right \vert \ll  \frac c  {\sqrt Z}
\end{equation*}
In this range,  we have 
\begin{equation*}
\mathcal H h  (\left \vert n/d^2p^{\mathfrak c_2} \right \vert  )\ll 
\begin{dcases} 
\frac {Z}  {c} , &\quad  \text{if} \ \sqrt Z \ll c \\
  c . &\quad  \text{if} \ \sqrt Z \gg c 
\end{dcases}
\end{equation*}
We set 
\begin{equation}\label{congruence}
d_n = (n+p^{2\mathfrak c_2 -2k}, d),
\end{equation} then 
\begin{equation*}
     \sum _{y\in (\mathbb Z/ d\mathbb Z)^\times } e\left ( \frac { (n+p^{2\mathfrak c_2 -2k})y} d  \right ) \ll d_n N^{o(1)}.
\end{equation*} 
Set $L = \sqrt Z / c$. Now the number of $n$ satisfy the bound and the congruence condition (\ref{congruence}) can be controlled by 
\begin{equation*}
\frac  {d^2 p^{\mathfrak c_2}(1+L)^2} {d_nZ}\ll \begin{dcases} 
\frac  {d^2 p^{\mathfrak c_2}} {d_nZ} , &\quad  \text{if} \ \sqrt Z \ll c \\
\frac {p^{\mathfrak c_2-2k}} {d_n} . &\quad  \text{if} \ \sqrt Z \gg c 
\end{dcases}
\end{equation*}
Then we have if $\sqrt Z \ll c$, 
\begin{equation*}
A_{c,Z} \ll \sum _{d_n| d} \frac {p^{k-\mathfrak c_2/2}} d \times \frac  {d^2 p^{\mathfrak c_2}} {d_nZ} \times d_n \times p^{-k/2} \times \frac Z c \ll p^{\mathfrak c_2/2-k/2},
\end{equation*}
and if $\sqrt Z \gg c$, 

\begin{equation*}
A_{c,Z} \ll \sum _{d_n| d} \frac {p^{k-\mathfrak c_2/2}} d \times \frac  {p^{\mathfrak c_2-2k}} {d_n} \times d_n \times p^{-k/2} \times c \ll  p^{\mathfrak c_2/2-k/2}.
\end{equation*}

\textbf{Case 2}. $\mathfrak c_2/2 \le k \le \mathfrak c_2$. Then 
\begin{equation} \label{a2}
    \begin{split}
       A_{c,Z} &= O(d^{-1}) \sum _{ n\in \mathbb N^*} \frac {\lambda _\pi (n)} {|n|_p^{1/2}} \sum _{y\in (\mathbb Z/ d\mathbb Z)^\times } e\left ( \frac { (1-n)y} {d}  \right )\\
       &\int _{\mathfrak o_p^\times }
  W_1 \left ( a\left (\frac 1 {c^2} \right ) \omega n\left (\frac u {\varpi^k} \right ) \right ) 
 W_2^{(k)} \left ( -\frac n {d^2u} \right ) \psi _p \left (  \frac {np^k} {c^2u} \right )du 
 \\
 &\mathcal H h \left (\left \vert \frac { n} {c^2} \right \vert \right ). 
    \end{split}
\end{equation}

Also by Lemma \ref{whittaker} and $c = p^k d$ we have 
\begin{equation*}
\begin{split}
W_1\left ( a\left (\frac 1 {c^2} \right ) \omega n\left (\frac u {\varpi^k} \right ) \right )= \psi_p\left ( - \frac {p^k} {c^2 u } \right ) W_1^{(k)} \left ( \frac 1 {d^2 u} \right ).
\end{split}
\end{equation*}

Then by Proposition \ref{mainlocal}, the integral in (\ref{a2}) is nonvanishing only if $\nu _p(1- n) = 2k-\mathfrak c_2$, which implies $\nu_p(n) = 0$. So we have 
\begin{equation*} 
    \begin{split}
       A_{c,Z} &= O(d^{-1}) \sum _{ n\in \mathbb N^*, (n,p)=1} \lambda _\pi (n) \sum _{y\in (\mathbb Z/ d\mathbb Z)^\times } e\left ( \frac {\overline {p^{2k}} (1- n)y} {d}  \right )\\
       & \left [ \int _{\mathfrak o_p^\times }
  W_1^{(k)} \left ( u  \right )
 W_2 ^{(k)}\left ( -u\right ) \psi_p \left ( \frac {(1-n) u } {p^k} \right ) du \right ]
 \mathcal H h \left (\left \vert \frac { n} {c^2} \right \vert \right ). 
    \end{split}
\end{equation*}

Now we bound terms in $A_{c,Z}$ one by one. By Proposition \ref{mainlocal} we have 
\begin{equation*}
\int _{\mathfrak o_p^\times }
  W_1^{(k)} ( u ) 
 W_2 ^{(k)} (-u ) \psi _p \left (\frac {(1-n)u} {p^k} \right )du \ll p^{k/2-\mathfrak c_2/2+O(1)}.
\end{equation*}
If  we set $(1-n,d) = d_n$, then 
\begin{equation*}
     \sum _{y\in (\mathbb Z/ d\mathbb Z)^\times } e\left ( \frac { (1-n)y} d  \right ) \ll d_n N^{o(1)}.
\end{equation*} 

And by definition 
\begin{equation*}
\begin{split}
\mathcal H h  (\left \vert n/c^2 \right \vert  )  &= \int _0 ^\infty \frac {\eta \left (  \frac x Z \right )} {\sqrt x} J_{\kappa -1} \left ( \frac {4\pi\sqrt x } c\right ) J_{k-1}\left ( \frac {4\pi\sqrt {nx} } c\right ) dx \\
&= \sqrt Z\int _0 ^\infty  \frac {\eta  (   x   )} {\sqrt x} J_{\kappa -1} \left ( \frac {4\pi\sqrt {Zx} } c\right ) J_{k-1}\left ( \frac {4\pi\sqrt {nZx} } c\right ) dx \\
&= \sqrt Z I\left (\frac {\sqrt Z} c, \frac {\sqrt {Zn}} c \right ).
\end{split}
\end{equation*}

Then by Lemma \ref{arch}, the function $\mathcal H h  (\left \vert n/c^2 \right \vert  )$  restricts the sum to essentially  $\vert 1 -\sqrt n \vert \ll  c / \sqrt Z$.
In this range we have 
\begin{equation*}
\mathcal H h  (\left \vert n/c^2 \right \vert  )\ll \begin{dcases}
c \ \ &\mathrm{if} \ \sqrt Z \gg c \\
 \frac Z c \ \ &\mathrm{if} \ \sqrt Z \ll c.
\end{dcases}
\end{equation*}

Note that  in our case we have  $\sqrt Z \ll p^{\mathfrak c_2/2} \le  c$. For fixed $d_n$, there is at most $\frac {c^2p^{\mathfrak c_2}} {d_n Zp^{2k}}$'s $n$ such that  $(1-n,d_p)=d_n$, $\nu _p(n-1)=2k -\mathfrak c_2$ and $\vert 1 -\sqrt n \vert \ll \frac c {\sqrt Z}$. Then we have
\begin{equation*} 
    \begin{split}
    A_{c,Z}&\ll   d \sum_{d_n | d}  \frac {c^2p^{\mathfrak c_2}} {d_n Zp^{2k}}\times d_n \times p^{k /2-\mathfrak c_2/2}\times \frac Z c \ll p^{\mathfrak c_2/2- k /2 } 
    \end{split}
\end{equation*}
\end{proof}

\begin{lemma}\label{MOD}
For $M^{od}$ as in (\ref{Mod}), we have 
\begin{equation*}
  M^{od} \ll  p^{\mathfrak c_2/2-\mathfrak c_1/4-l/2}. 
\end{equation*}
\end{lemma}

\begin{proof}
By the definition and Proposition \ref{Acz} we have 
\begin{equation*}
\begin{split}
M^{od} & =p^{\mathfrak c_1/2+l} \sum _{c_l | c} \sum_Z  \frac {A_{c,Z}} c \\
& \ll p^{\mathfrak c_1/2+l} \sum _{k\ge \mathfrak c_1/2+l} \sum_{d\ge 1} \sum _Z \frac {p^{\mathfrak c_2/2-3k/2 }} d \ll p^{\mathfrak c_2/2-\mathfrak c_1/4-l/2}.
\end{split}
\end{equation*}
\end{proof}

\begin{theorem} 
Let $\mathcal F_\theta [l]$ be the set of holomorphic new forms of weight $\kappa \ge 4$, level $p^{\mathfrak c_1}$ and trivial nebentypus,whose associated local component $\pi_p$ belongs to a small family $\pi _\theta [l]$ as above. Let $g$ be a cusp newform with level $N =p^{\mathfrak c_2}>  p^{\mathfrak c_1}$,fixed weight $\kappa _g \ge 4$,and trivial nebentypus.Then we have
\begin{equation}\label{firstmoment}    
\sum_{f\in \mathcal F_\theta [l]} \frac  {L(1/2, f\times g)} {\Vert f \Vert ^2} \ll_{p,\varepsilon} N^{\varepsilon} \left (p^{\mathfrak c_1/2+l} + \frac {p ^{\mathfrak c_2/2 }} {p^{\mathfrak c_1/4+l/2}}\right ).
\end{equation}
Furthermore  suppose $L(f\times g,1/2) \ge 0$ for all $f\in \mathcal F_\theta [l]$. Suppose $\mathfrak c_1 = \delta \mathfrak c_2$ for $0<\delta<1$. By picking $l= l_0 \in \{ 0, 1\}$, we get the subconvex bound 
\begin{equation*}
L(1/2,f\times g) \ll N^{\max \{ \frac \delta  2, \frac {1- \delta }  2  \}+\epsilon}.
\end{equation*}
By picking  $l$ to be the closest integer to $3\mathfrak c_2/2 -\mathfrak c_1/2$
  while $1\le l <i_0$, we get that
\begin{equation*}
 L(1/2, f\times g) \ll_{p,\varepsilon} N^{\max \left \{\frac {1-\delta} 2, \frac 1 3, \frac \delta 2 \right \}+ \varepsilon}.
\end{equation*}
  In particular we obtain a hybrid subconvexity bound for $\delta$  in any compact subset of $(0,1)$, which is further more a Weyl bound in the range $1/3 \le  \delta  \le  2/3$.
\end{theorem}

\begin{proof}
    By \eqref{M} and \eqref{M0}, the left side of \eqref{firstmoment} is bounded by 
    \[
        M^d + M^{od}.
    \]
    Then by Theorem \ref{trace formula} we have $M^d \ll p^{\mathfrak c_1/2+l}$, and by Lemma $\ref{MOD}$ we have $M^{od}\ll p^{\mathfrak c_2/2 -\mathfrak c_1/4 -l/2}$.
\end{proof}

\textbf{Acknowledgement.} We would like to thank Yueke Hu, Liyuan Ye, Yuhao Cheng and Wenbo Liu for helpful feedback on an earlier draft.


\newcommand{\etalchar}[1]{$^{#1}$}
\def\cprime{$'$} \def\cprime{$'$} \def\cprime{$'$} \def\cprime{$'$}

\end{document}